\documentclass[12pt]{amsart}
\usepackage{epsfig}
\usepackage{amsmath,amssymb,amsthm,bm}
\usepackage{times}

\usepackage{ifthen}

\newcommand{\T}{\mathbb{T}}
\newcommand{\Sn}{\mathbb S^{n-1}}
\newcommand{\intd}{\mathrm{d}}
\newcommand{\MinTen}[3]
  {
    \ifthenelse{\equal{#2}{0}}{\Phi_{#1}^{#3}}{\Phi_{#1}^{#2,#3}}
  }
\newcommand{\MinTenZ}[3]
  {
    \MinTen{#1}{{#2}}{#3}
  }

\newcommand{\IntMinTen}[4]
  {
    \ifthenelse{\equal{#2}{0}}{\Phi_{#1,#4}^{#3}}{\Phi_{#1,#4}^{#2,#3}}
  }
\newcommand{\IntMinTenZ}[4]
  {
    \IntMinTen{#1}{{#2}}{#3}{#4}
  }

\newcommand{\TenCM}[4]
  {
    \phi_{#1}^{#2,#3,#4}
  }

\newcommand{\IntTenCM}[5]
  {
    \phi_{#1,#5}^{#2,#3,#4}
  }

\newcommand{\PTenCM}[4]
  {
    \psi_{#1}^{#2,#3,#4}
  }

\DeclareMathOperator{\Nor}{{\rm Nor}}
\DeclareMathOperator{\sign}{{\rm sign}}

\newcommand{\eps}{\varepsilon}
\newcommand{\1}{{\mathbf{1}}} 

\newcommand{\cB}{{\mathcal{B}}}

\newcommand{\cF}{{\mathcal{F}}}

\newcommand{\cH}{{\mathcal{H}}}

\newcommand{\cK}{{\mathcal{K}}}

\newcommand{\cP}{{\mathcal{P}}}

\newcommand{\N}{{\mathbb{N}}} 
\newcommand{\R}{{\mathbb{R}}} 

\DeclareMathOperator{\Span}{{\rm span}}

\usepackage{mathtools}

\DeclarePairedDelimiter\norm\lVert\rVert
\DeclarePairedDelimiter\abs\lvert\rvert
\newcommand\nularg{\,\cdot\,}

\newtheorem{lemma}{Lemma}	
\newtheorem{corollary}{Corollary}
\newtheorem{theorem}{Theorem}
\newtheorem{proposition}{Proposition}	

\newtheorem*{lemma*}{Lemma}

\theoremstyle{definition}

\begin{document}
\bigskip

\title[Crofton Formulae for Tensor-Valued Curvature Measures]{Crofton Formulae for Tensor-Valued Curvature Measures}

\author[Daniel Hug and Jan A. Weis]
{Daniel Hug and Jan A. Weis}
\address{Karlsruhe Institute of Technology (KIT), Department of Mathematics, D-76128 Karlsruhe, Germany} \email{daniel.hug@kit.edu}
\address{Karlsruhe Institute of Technology (KIT), Department of Mathematics, D-76128 Karlsruhe, Germany}
\email{jan.weis@kit.edu}
\thanks{The authors were supported in part by DFG grants FOR 1548 and HU 1874/4-2}

\pagestyle{myheadings}
\markboth{Daniel Hug and Jan A. Weis}{Crofton Formulae for Tensor-Valued Curvature Measures}

\begin{abstract}
    The tensorial curvature measures are tensor-valued
    generalizations of the curvature measures of convex bodies. We
    prove a set of Crofton formulae for such tensorial curvature
    measures.  These formulae express the integral mean of the
    tensorial curvature measures of the intersection of a given convex
    body with a uniform affine $k$-flat in terms of linear
    combinations of tensorial curvature measures of the given convex
    body. Here we first focus on the case where the tensorial curvature
    measures of the intersection of the given body with an affine flat
    is defined with respect to the affine flat as its ambient space.
    From these formulae we then deduce some new and also recover known
    special cases. In particular, we substantially simplify some of
    the constants that were obained in previous work on Minkowski
    tensors. In a second step, we explain how the results can be extended
		to the case where the tensorial curvature
    measure of the intersection of the given body with an affine flat
    is determined with respect to the ambient Euclidean space.
\end{abstract}

\maketitle

  \section{Introduction}\label{14-sec1}

  The \emph{classical Crofton formula} is a major result in integral
  geometry.  Its name originates from works of the Irish mathematician
  Morgan W. Crofton \cite{14-Crofton85} on integral geometry in
  $\R^{2}$ in the late 19th century. For a convex body $K$ (a
  non-empty, convex and compact set) in the $n$-dimensional Euclidean
  space $\R^{n}$, $n \in \N$, the classical Crofton formula (see
  \cite[(4.59)]{14-Schneider14}) states that
  \begin{equation} \label{14-Form_Croft_Classic} \int_{A(n, k)} V_{j}
    (K \cap E) \, \mu_{k} (\intd E) = \alpha_{n j k} V_{n - k + j}
    (K),
  \end{equation}
  for $k \in \{ 0, \ldots, n \}$ and $j \in \{ 0, \ldots, k \}$, where
  $A(n, k)$ is the affine Grassmannian of $k$-flats in $\R^{n}$,
  $\mu_{k}$ denotes the motion invariant Haar measure on $A(n,k)$,
  normalized as in \cite[p. 588]{14-SchnWeil08}, and $\alpha_{n j k} >
  0$ is an explicitly known constant.

  Let $ \cK^{n} $ denote the set of convex bodies in $\R^n$.  The
  functionals $V_{i} :\cK^n\rightarrow \R$, for $i\in\{0,\ldots,n\}$,
  appearing in \eqref{14-Form_Croft_Classic}, are the {\em intrinsic
    volumes}, which occur as the coefficients of the monomials in the
  \emph{Steiner formula}
  \begin{equation} \label{14-Form_Steiner} V_{n} (K + \epsilon B^{n})
    = \sum_{j = 0}^{n} \kappa_{n - j} V_{j} (K) \epsilon^{n - j},
  \end{equation}
  for a convex body $K \in \cK^{n}$ and $\epsilon \geq 0$
  (cf. \cite[(1.16)]{14-KidVed16}); here, as usual, $+$ denotes the Minkowski
  addition in $\R^{n}$ and $\kappa_{n}$ is the volume of
	the Euclidean unit ball $B^{n}$  in $\R^{n}$.  Properties of the
  $V_i$ such as continuity, isometry invariance and additivity are
  derived from corresponding properties of the volume.  A key result
  for the intrinsic volumes is \emph{Hadwiger's characterization
    theorem} (see \cite[2. Satz]{14-Hadwiger52} and \cite[Theorem~1.23]{14-KidVed16}
), which states that $V_{0}, \ldots, V_{n}$
  form a basis of the vector space of continuous and isometry
  invariant real-valued valuations on $\cK^{n}$.

  A natural way to extend the Crofton formula is to apply the
  integration over the affine Grassmannian $A(n, k)$ to functionals
  which generalize the intrinsic volumes. One of these generalizations
  concerns the class of continuous and isometry covariant
  $\T^{p}$-valued valuations on $\cK^{n}$, where $\T^{p}$ denotes the
  vector space of symmetric tensors of rank $p \in \N_{0}$ over
  $\R^n$.

  The $\T^{0}$-valued valuations are simply the well-known and
  extensively studied intrinsic volumes. For the $\T^{1}$-valued
  (i.e. vector-valued) valuations, Hadwiger and Schneider
  \cite[Hauptsatz]{14-HadSchn71} proved in 1971 a characterization
  theorem similar to the aforementioned real-valued case due to
  Hadwiger. In addition, they also established integral geometric
  formulae, including a Crofton formula \cite[(5.4)]{14-HadSchn71}. In
  1997, McMullen \cite{14-McMullen97} initiated a systematic
  investigation of this class of $\T^{p}$-valued valuations for
  general $p \in \N_0$.  Only two years later Alesker generalized
  Hadwiger's characterization theorem (see \cite[Theorem
  2.2]{14-Alesker99b} and \cite[Theorem~2.5]{14-KidVed16}) by showing that the
  vector space of continuous and isometry covariant $\T^{p}$-valued
  valuations on $\cK^{n}$ is spanned by the tensor-valued versions of
  the intrinsic volumes, the \emph{Minkowski tensors}
  $\MinTen{j}{r}{s}$, where $j, r, s \in \N_{0}$ and $j < n$,
  multiplied with suitable powers of the metric tensor in~$\R^{n}$.
  In 2008, Hug, Schneider and Schuster proved a set of Crofton
  formulae for these Minkowski tensors (see \cite[Theorem
  2.1--2.6]{14-HugSchnSchu08}).

  Localizations of the intrinsic volumes yield other types of
  generalizations. The \emph{support measures} are weakly continuous,
  locally defined and motion equivariant valuations on convex bodies
  with values in the space of finite measures on Borel subsets of
  $\R^{n} \times \Sn$, where $\Sn$ denotes the Euclidean unit sphere
  in $\R^{n}$. These are determined by a local version of
  \eqref{14-Form_Steiner}. Therefore, they are a crucial example of
  localizations of the intrinsic volumes. Furthermore, their marginal
  measures on Borel subsets of $\R^{n}$ are called \emph{curvature
    measures} and the ones on Borel subsets of $\Sn$ are called
  \emph{surface area measures}. In 1959, Federer \cite{14-Federer59}
  proved Crofton formulae for curvature measures, even in the more
  general setting of sets with positive reach. For~further details and
  references, see also \cite[Section~1.3]{14-KidVed16} and
  \cite[Section~1.5]{14-KidVed16}.  Certain Crofton formulae for support
  measures were proved by Glasauer in 1997
  \cite[Theorem~3.2]{14-Glasauer97}.

  The combination of Minkowski tensors and localizations leads to
  another generalization of the intrinsic volumes. This topic has been
  explored by Schneider \cite{14-Schneider13} and Hug and Schneider
  \cite{14-HugSchn14, 14-HugSchn16} in recent years. They introduced
  particular tensorial support measures, the \emph{generalized local
    Minkowski tensors}, and proved that they essentially span the
  vector space of isometry covariant and locally defined valuations on
  the space of convex polytopes $\cP^{n}$ with values in the
  $\T^{p}$-valued measures on $\cB(\R^{n} \times \Sn)$ (see \cite[Section~2.4]{14-KidVed16}).
	Under the additional assumption of weak
  continuity they extended this result to valuations on $\cK^{n}$; a
  summary of the required arguments is given in \cite[Section~2.5]{14-KidVed16}.

  The aim of the present article is to prove a set of Crofton formulae
  for similar functionals, which are localized in $\R^{n}$, the
  \emph{tensorial curvature measures} or \emph{tensor-valued curvature measures}.
	Here we first focus on the case where
  the tensorial curvature measures of the intersection of the given
  body with an affine flat are defined with respect to the affine flat
  as the ambient space (intrinsic viewpoint). In a second step, we
	demonstrate how the arguments
  can be extended to the case where the curvature measures are
  considered in $\R^n$ (extrinsic viewpoint).  The current approach combines main ideas of
  the previous works \cite{14-HugSchnSchu08} and
  \cite{14-HugSchn14} and also links it to \cite{14-BernHug15}.
	A major advantage of the localization is that
  it naturally leads to a suitable choice of local tensor-valued
  measures for which the constants in the Crofton formulae are
  reasonably simple. From the general local results, we finally deduce
  various special consequences for the total measures, which are the
  Minkowski tensors that have been studied in
  \cite{14-HugSchnSchu08}. For the latter, we restrict ourselves to
  the translation invariant case, which simplifies the involved
  constants, but the general case can be treated similarly.
	In the case of the results for the extrinsic tensorial Crofton formulae,
	the connection to the approach in \cite{14-BernHug15} via the methods of
	algebraic integral geometry is used and deepened, but this interplay
	will have to be explored further in future work.
	
	The structure of this contribution is as follows. In Section \ref{14-sce2}, we
	fix our notation and
	collect various auxiliary results which will be needed. Section \ref{14-sec3}
	contains the main results. We first state our findings for intrinsic tensorial
	curvature measures, then discuss some special cases and finally explain
	the extension to extrinsic tensorial curvature measures. The proofs of the results
	for the intrinsic case are given in Section \ref{14-sec4}. Section \ref{14-sec5new}
	contains the arguments in the extrinsic setting. Some auxiliary results on sums
	of Gamma functions are provided in the final section.

  \section{Some Basic Tools}\label{14-sce2}

  In the following, we work in the $n$-dimensional Euclidean space
  $\R^{n}$, equipped with its usual topology generated by the standard
  scalar product $\cdot$ and the corresponding Euclidean norm
  $\norm{\nularg}$. Recall that the unit ball centered at the origin is denoted by $B^n$,
	its boundary (the unit sphere) is denoted by $\Sn$.
	For a topological space $X$, we denote the Borel
  $\sigma$-algebra on $X$ by $\cB(X)$.

  By $G(n, k)$, for $k \in \{0, \ldots, n\}$, we denote the
  Grassmannian of $k$-dimensional linear subspaces in $\R^{n}$, and we
  write $\nu_k$ for the (rotation invariant) Haar probability measure
  on $G(n, k)$. The directional space of an affine $k$-flat $E \in
  A(n, k)$ is denoted by $L(E) \in G(n, k)$, its orthogonal complement
  by $E^{\perp} \in G(n, n - k)$, and the translate of $E$ by a vector
  $t \in \R^{n}$ is denoted by $E_{t} := E + t$. For $k \in \{0,
  \ldots, n\}$, $l \in \{0, \ldots k\}$ and $F \in G(n, k)$, we define
  $G(F, l) := \{ L \in G(n, l): L \subset F \}$.  On $G(F, l)$ there
  exists a unique Haar probability measure $\nu^{F}_l$ invariant under
  rotations of $\R^n$ mapping $F$ into itself and leaving $F^{\perp}$
  pointwise fixed. The orthogonal projection of a vector $x \in
  \R^{n}$ to a linear subspace $L$ of $\R^{n}$ is denoted by
  $p_{L}(x)$ and its direction by $\pi_{L}(x) \in \Sn$, if $x \notin
  L^{\perp}$. For two linear subspaces $L, L'$ of $\R^{n}$, the
  generalized sine function $[L, L']$ is defined as follows. One
  extends an orthonormal basis of $L \cap L'$ to an orthonormal basis
  of $L$ and to one of $L'$. Then $[L, L']$ is the volume of the
  parallelepiped spanned by all these vectors.

  The vector space of symmetric tensors of rank $p \in \N_{0}$ over
  $\R^n$ is denoted by~$\T^{p}$. The symmetric tensor product of two
  vectors $x, y \in \R^{n}$ is denoted by $xy$ and the $p$-fold tensor
  product of a vector $x \in \R^{n}$ by $x^{p}$. Identifying $\R^{n}$
  with its dual space via its scalar product, we interpret a symmetric
  tensor $a \in \T^{p}$ as a symmetric $p$-linear map from
  $(\R^{n})^{p}$ to $\R$. One special tensor is the \emph{metric
    tensor} $Q \in \T^{2}$, defined by $Q(x, y) := x \cdot y$ for $x,
  y \in \R^{n}$. For an affine $k$-flat $E \in A(n, k)$, $k \in \{0,
  \ldots, n\}$, the metric tensor $Q(E)$ in $E$ is defined by $Q(E)(x,
  y) := p_{L(E)} (x) \cdot p_{L(E)} (y)$ for $x, y \in \R^{n}$.

  Defining the tensorial curvature measures requires some preparation
  (see also \cite[Section~1.3]{14-KidVed16}). For a convex body $K \in \cK^{n}$,
  we call the pair $(x, u) \in \R^{2n}$ a \emph{support element}
  whenever $x$ is a boundary point of $K$ and $u$ is an outer unit
  normal vector of $K$ at $x$. The set of all these support elements
  of $K$ is denoted by $\Nor K \subset \Sigma^{n} := \R^{n} \times
  \Sn$ and called the \emph{normal bundle} of $K$.
	For $x \in \R^{n}$, we denote the
  metric projection of $x$ onto $K$ by $p(K, x)$, and define $u(K, x)
  := (x - p(K, x)) / \norm{x - p(K, x)}$ for $x \in \R^{n} \setminus
  K$, the unit vector pointing from $p(K, x)$ to $x$. For $\epsilon >
  0$ and a Borel set $\eta \subset \Sigma^{n}$,
  \begin{equation*}
    M_{\epsilon}(K, \eta) := \left\{ x \in \left( K + \epsilon B^{n} \right) \setminus K \colon \left( p(K, x), u(K, x) \right) \in \eta \right\}
  \end{equation*}
  is a local parallel set of $K$ which satisfies a \emph{local Steiner
    formula}
  \begin{equation} \label{14-Form_Steiner_loc} V_{n} (M_{\epsilon}(K,
    \eta)) = \sum_{j = 0}^{n - 1} \kappa_{n - j} \Lambda_{j} (K, \eta)
    \epsilon^{n - j}, \qquad \epsilon \geq 0.
  \end{equation}
  This relation determines the \emph{support measures} $\Lambda_{0}
  (K, \cdot), \ldots, \Lambda_{n - 1} (K, \cdot)$ of $K$, which are
  finite Borel measures on $\cB (\Sigma^{n})$. Obviously, a comparison
  of \eqref{14-Form_Steiner} and \eqref{14-Form_Steiner_loc} yields
  $V_{j}(K) = \Lambda_{j} (K, \Sigma^{n})$.

  Now, for a convex body $K \in \cK^{n}$, a Borel set $\beta \in
  \cB(\R^{n})$ and $j, r, s \in \N_{0}$, the \emph{tensorial curvature
    measures} are given by
  \begin{equation*}
    \TenCM{j}{r}{s}{0} (K, \beta):= \omega_{n - j} \int _{\beta \times \Sn} x^r u^s \, \Lambda_j(K, \intd (x, u)),
  \end{equation*}
  for $j\in\{0,\ldots,n-1\}$, where $\omega_{n}$ denotes the
  $n-1$-dimensional volume of $\Sn$, and by
  \begin{equation*}
    \TenCM{n}{r}{0}{0} (K, \beta): = \int _{K \cap \beta} x^r \, \cH^{n}(\intd x).
  \end{equation*}
  If $K \subset E \in A(n, k)$ with $j < k \leq n$, we denote the
  $j$-th support measure of $K$ defined with respect to $E$ as the
  ambient space by $\Lambda^{(E)}_j(K, \cdot)$, which is a Borel
  measure on $\cB(\R^{n} \times (L(E) \cap \Sn))$, concentrated on
  $\Sigma^{(E)}:= E \times (L(E) \cap \Sn)$ with $L(E) \in G(n, k)$
  being the linear subspace parallel to $E$.  Then, we define the
  intrinsic tensorial curvature measures
  \begin{equation*}
    \IntTenCM{j}{r}{s}{0}{E} (K, \beta): = \omega_{k - j} \int _{\beta \times (L(E) \cap \Sn)} x^r u^s \, \Lambda^{(E)}_j(K, \intd (x, u))
  \end{equation*}
  and
  \begin{equation*}
    \IntTenCM{k}{r}{0}{0}{E} (K, \beta) := \int _{K \cap \beta} x^r \, \cH^{k}(\intd x).
  \end{equation*}
  For the sake of convenience, we extend the definition by
  $\TenCM{j}{r}{s}{0} := 0$ (resp. $\IntTenCM{j}{r}{s}{0}{E} := 0$)
  for $j \notin \{ 0, \ldots, n \}$ (resp. $j \notin \{ 0, \ldots, k
  \}$) or $r \notin \N_{0}$ or $s \notin \N_{0}$ or $j = n$ (resp. $j
  = k$) and $s \neq 0$. We adopt the same convention for the Minkowski
  tensors and the generalized tensorial curvature measures introduced
  below.

  The tensorial curvature measures are natural local versions of the
  \emph{Minkowski tensors}. For a convex body $K \in \cK^{n}$ and $j,
  r, s \in \N_{0}$, the latter are just the total measures
  $\smash{\MinTen{j}{r}{s}} (K) := \smash{\TenCM{j}{r}{s}{0}} (K,
  \R^{n})$ and, if $K \subset E \in A(n, k)$, an intrinsic version is
  given by $\smash{\IntMinTen{j}{r}{s}{E}} (K) :=
  \smash{\IntTenCM{j}{r}{s}{0}{E}} (K, \R^{n})$. These definitions of
  the Minkowski tensors differ slightly from the ones commonly used in
  the literature, as we slightly change the usual normalization
  (compare with the normalization used in \cite[Definition~2.1]{14-KidVed16}). The purpose of this change is to
  simplify the presentation of the main results of this article (and
  of future work).

  For a polytope $P \in \cP^{n}$ and $j \in \{ 0, \ldots, n \}$, we
  denote the set of $j$-dimensional faces of $P$ by $\cF_{j}(P)$ and
  the normal cone of $P$ at a face $F \in \cF_{j}(P)$ by $N(P,F)$. For
  a polytope $P \in \cP^{n}$ and a Borel set $\eta\subset\Sigma$, the $j$-th support measure is explicitly
  given by
  \begin{equation*}
    \Lambda_{j} (P, \eta) = \frac {1} {\omega_{n - j}}
    \sum_{F \in \cF_{j}(P)} \int _{F} \int _{N(P,F) \cap \Sn}
    \1_\eta(x,u)
    \, \cH^{n - j - 1} (\intd u) \, \cH^{j} (\intd x)
  \end{equation*}
  for $j\in\{0,\ldots,n-1\}$.  For $\beta \in \cB(\R^n)$, this yields
  \begin{equation*}
    \TenCM{j}{r}{s}{0} (P, \beta) = \sum_{F \in \cF_{j}(P)} \int _{F \cap \beta} x^r \, \cH^{j}(\intd x) \int _{N(P,F) \cap \Sn} u^{s} \, \cH^{n - j - 1} (\intd u)
  \end{equation*}
  and, if $P \subset E \in A(n, k)$ and $j < k \leq n$,
  \begin{equation*}
    \IntTenCM{j}{r}{s}{0}{E} (P, \beta) = \sum_{F \in \cF_{j}(P)} \int _{F \cap \beta} x^r \, \cH^{j}(\intd x) \int _{N_{E}(P,F) \cap \Sn} u^{s} \, \cH^{k - j - 1} (\intd u),
  \end{equation*}
  where $N_{E}(P, F) = N(P, F) \cap L(E)$ is the normal cone of $P$ at
  the face $F$, taken with respect to the subspace $L(E)$.  Of course,
  analogous representations are obtained for the (global) intrinsic
  and extrinsic Minkowski tensors.

  The Crofton formulae, which are stated in the next section, will
  naturally also involve the \emph{generalized tensorial curvature
    measures} (see \cite[(2.28)]{14-KidVed16})
  \begin{equation*}
    \TenCM{j}{r}{s}{1} (P, \beta) := \sum_{F \in \cF_{j}(P)} Q(F) \int _{F \cap \beta} x^r \, \cH^{j}(\intd x) \int _{N(P,F) \cap \Sn} u^{s} \, \cH^{n - j - 1} (\intd u),
  \end{equation*}
  for $j \in\{1,\ldots,n-1\}$, and, if $P \subset E \in A(n, k)$ and
  $0 < j < k \leq n$,
  \begin{equation*}
    \IntTenCM{j}{r}{s}{1}{E} (P, \beta) := \sum_{F \in \cF_{j}(P)} Q(F)
		\int _{F \cap \beta} x^r \, \cH^{j}(\intd x) \int _{N_{E}(P,F)
		\cap \Sn} u^{s} \, \cH^{k - j - 1} (\intd u).
  \end{equation*}

  Due to Hug and Schneider \cite{14-HugSchn14} there exists a weakly
  continuous extension of the generalized tensorial curvature measures
  to $\cK^{n}$. In fact, they proved such an extension for the
  generalized local Minkowski tensors, which are measures on
  $\cB(\R^{n} \times \Sn)$. Globalizing this in the $\Sn$-coordinate
  yields the result for the tensorial curvature measures.

  Apart from the easily verified relation
  \begin{align} \label{14-Form_Rel_tenCM} \TenCM{n - 1}{r}{s}{1} = Q
    \TenCM{n - 1}{r}{s}{0} - \TenCM{n - 1}{r}{s + 2}{0},
  \end{align}
  the tensorial curvature measures and the generalized tensorial
  curvature measures are linearly independent.  In contrast, McMullen
  \cite{14-McMullen97} discovered basic linear relations for the
  (gobal) Minkowski tensors (see also \cite[Theorem~2.6]{14-KidVed16}), and it
  was shown in \cite{14-HSS07a} that these are essentially all linear
  dependences between the Minkowski tensors (see also \cite[Theorem~2.7]{14-KidVed16}). Furthermore, McMullen
  \cite[p.~269]{14-McMullen97} found relations for the global
  counterparts of the generalized tensorial curvature measures. In
  fact, the globalized form of \eqref{14-Form_Rel_tenCM} is a very
  special example of one of these relations. For the translation
  invariant Minkowski tensors $\MinTenZ{j}{0}{s} $, these relations
  take a very simple form, nevertheless for our purpose they are
  essential in the proof of Theorem~\ref{14-Thm_Main_Gen}. To have a
  short notation for these translation invariant Minkowski tensors, we
  omit the first superscript and put
  \begin{equation*}
    \MinTen{j}{0}{s} := \MinTenZ{j}{0}{s}, \qquad \IntMinTen{j}{0}{s}{E}
		:= \IntMinTenZ{j}{0}{s}{E}.
  \end{equation*}
  Then we can state the following very special case of McMullen's
  relations.

  \begin{lemma}[McMullen] \label{14-Lem_McMullen} Let $P \in \cP^n$
    and $j, s \in \N_0$ with $j \leq n - 1$. Then
    \begin{equation*}
      \tfrac {n - j + s} {s + 1} \MinTen{j}{0}{s + 2}(P)
      = \sum_{F \in \cF_j (P)} Q (F^{\perp}) \cH^{j} (F) \int _{N(P,F)
        \cap \Sn} u^{s}
      \, \cH^{n - j - 1} (\intd u).
    \end{equation*}
  \end{lemma}

  Note that this lemma is essentially a global result which is derived
  by applying a version of the divergence theorem.

  \section{Crofton Formulae}\label{14-sec3}

  In this article, for $0 \leq j \leq k < n$ and $i, s \in \N_0$, we
  are first concerned with the Crofton integrals
  \begin{equation} \label{14-Form_Main_Loc} \int_{A(n, k)} Q(E)^i \,
    \IntTenCM{j}{r}{s}{0}{E} (K \cap E, \beta \cap E) \, \mu_k(\intd
    E),
  \end{equation}
  which involve the intrinsic tensorial curvature measures, and the
  Crofton integrals
  \begin{equation} \label{14-Form_Main} \int_{A(n, k)} Q(E)^i \,
    \IntMinTen{j}{0}{s}{E} (K \cap E) \, \mu_k(\intd E)
  \end{equation}
  for the global versions of the translation invariant intrinsic
  tensorial curvature measures, the translation invariant intrinsic
  Minkowski tensors obtained by setting $r=0$. In the global case, we
  restrict our investigations mainly to these translation invariant
  intrinsic Minkowski tensors, general Crofton formulae have already
  been established in \cite{14-HugSchnSchu08}.
	
	Using the
  simplifications of the formulae obtained in the present work, the
  extrinsic formulae in \cite{14-HugSchnSchu08}, that is, Crofton formulae for
	the integrals
	\begin{equation} \label{14-Form_Main2} \int_{A(n, k)}
    \TenCM{j}{r}{s}{0} (K \cap E,\beta\cap E) \, \mu_k(\intd E)
  \end{equation}
	can be simplified
  accordingly.
	We explain this in detail in the case where $j=k-1$. The connection to
	\cite{14-BernHug15} turns out to be crucial for simplifying the constants if
	$s$ is odd. However, for even $s$
	the current approach works completely independently.

  \subsection{Crofton Formulae for Intrinsic Tensorial Curvature
    Measures}\label{14-sec3.1}

  In this section we state the formulae for the integrals given in
  \eqref{14-Form_Main_Loc} and \eqref{14-Form_Main}. We start with the
  local versions, where we distinguish the cases $j = k$ and $j < k$.

  \begin{theorem} \label{14-Thm_Main_j=k} Let $K \in \cK^n$, $\beta
    \in \cB(\R^n)$ and $i, k, r, s \in \N_0$ with $k < n$. Then
    \begin{align*}
      & \int_{A(n, k)} Q(E)^i \, \IntTenCM{k}{r}{s}{0}{E} (K \cap E,
      \beta \cap E) \, \mu_k(\intd E) = \frac {\Gamma(\frac {n}
        {2})\Gamma(\frac {k} {2} + i)} {\Gamma(\frac {n} {2} +
        i)\Gamma(\frac {k} {2})} \, Q^i \TenCM{n}{r}{0}{0} (K, \beta)
    \end{align*}
    if $s=0$; for $s\neq 0$ the integral is zero.
  \end{theorem}

  If $s = 0$ in Theorem~\ref{14-Thm_Main_j=k}, then we interpret the
  coefficient of the tensor on the right-hand side as $0$, if $k = 0$
  and $i \neq 0$, and as $1$, if $k = i = 0$.  A global version of
  Theorem~\ref{14-Thm_Main_j=k} is obtained by simply setting $\beta =
  \R^n$.

Next we turn to case $j<k$.

  \begin{theorem} \label{14-Thm_Main_Gen_Loc} Let $K \in \cK^n$,
    $\beta \in \cB(\R^n)$ and $i, j, k, r, s \in \N_0$ with $j < k <
    n$ and $k > 1$. Then
    \begin{align*}
      \MoveEqLeft[3]  \int_{A(n, k)} Q(E)^i \, \IntTenCM{j}{r}{s}{0}{E}
			(K \cap E, \beta \cap E) \, \mu_k(\intd E) \\
      ={}& \gamma_{n, k, j} \sum_{z = 0}^{\lfloor \frac s 2 \rfloor +
        i} Q^z \bigl( \lambda_{n, k, j, s, i, z}^{(0)} \, \TenCM{n - k
        + j}{r}{s + 2i - 2z}{0} (K, \beta) + \lambda_{n, k, j, s, i,
        z}^{(1)} \, \TenCM{n - k + j}{r}{s + 2i - 2z - 2}{1} (K,
      \beta) \bigr),
    \end{align*}
    where for $\varepsilon \in \{ 0, 1 \}$ we set
    \begin{align*}
      \gamma_{n, k, j} :={} & \binom{n - k + j - 1}{j}
      \frac{\Gamma(\frac {n - k + 1} 2)}{2 \pi},
      \\
      \lambda_{n, k, j, s, i, z}^{(\varepsilon)} := {} & \sum_{p =
        0}^{i} \sum_{q = (z - p + \varepsilon)^+}^{\lfloor \frac s 2
        \rfloor + i - p} (-1)^{p + q - z} \binom{i}{p} \binom{s + 2i -
        2p}{2q} \binom{p + q - \varepsilon}{z} \Gamma(q + \tfrac 1 2)
      \\
      & \times \frac {\Gamma(\frac {j + s} 2 + i - p - q + 1)}
      {\Gamma(\frac {n - k + j + s} 2 + i - p + 1)} \frac
      {\Gamma(\frac {k - 1} 2 + p) \Gamma(\frac {n - k} 2 + q)}
      {\Gamma(\frac {n + 1} 2 + p + q)} \vartheta_{n, k, j, p,
        q}^{(\varepsilon)},
      \\
      \vartheta_{n, k, j, p, q}^{(0)} := {} & (n - k + j) \big(\tfrac
      {k - 1} 2 + p\big), \qquad \vartheta_{n, k, j, p, q}^{(1)} :=
      p(n - k) - q(k - 1).
    \end{align*}
  \end{theorem}

  If $j=k-1$, then the tensorial curvature measures and the
  generalized tensorial curvature measures are linearly dependent. In
  this case, the right-hand side can be expressed as a linear
  combination of the tensor-valued curvature measures $Q^z\TenCM{n -
    1}{r}{s + 2i - 2z}{0} (K, \nularg )$, whereas the measures
  $Q^z\TenCM{n - 1}{r}{s + 2i - 2z}{1} (K, \nularg)$ are not
  needed. An explicit description of this case is given in
  Corollary~\ref{14-Cor_j=k-1} for $i=0$ and in \eqref{Form_CF_Main_j=k-1} for $i \in \N_{0}$.

  If the additional metric tensor is omitted as a weight function,
  that is in the case $i = 0=p$, then the coefficients $\lambda_{n, k,
    j, s, 0, z}^{\smash{(\varepsilon)}}$ in
  Theorem~\ref{14-Thm_Main_Gen_Loc} simplify to a single sum.

  Apparently, the coefficients in Theorem~\ref{14-Thm_Main_Gen_Loc}
  are not well defined in the (excluded) case $k = 1$ and $j=0$, as
  $\Gamma(0)$ is involved in the numerator of $\lambda_{n, 1, 0, s, i,
    z}^{\smash{(\varepsilon)}}$. Although this issue can be resolved
  by a proper interpretation of the (otherwise ambiguous) expression
  $\Gamma(p)\cdot p=\Gamma(p+1)$ as $1$ for $p=0$, we prefer to state
  and derive this case separately.  In fact, our analysis leads to
  substantial simplifications of the constants, as our next result
  shows.

  \begin{theorem} \label{14-Thm_Main_k=1} Let $K \in \cK^n$, $\beta
    \in \cB(\R^n)$ and $i, r, s \in \N_0$. Then
    \begin{align*}
      \MoveEqLeft  \int_{A(n, 1)} Q(E)^i \, \IntTenCM{0}{r}{s}{0}{E} (K \cap E, \beta \cap E) \, \mu_1(\intd E) \\
      & = \frac {\Gamma( \frac n 2 ) \Gamma(\frac {s + 1} 2 + i)} {
        \pi \Gamma ( \frac {n + s + 1} 2 + i) } \sum_{z = 0}^{\frac
        {s} {2} + i} (-1)^{z} \binom{\frac s 2 + i}{z} \frac {1} {1 -
        2z} \, Q^{\frac s 2 + i - z} \TenCM{n - 1}{r}{2z}{0} (K,
      \beta)
    \end{align*}
    for even $s$.  If $s$ is odd, then
    \begin{align*}
      \MoveEqLeft  \int_{A(n, 1)} Q(E)^i \, \IntTenCM{0}{r}{s}{0}{E} (K \cap E, \beta \cap E) \, \mu_1(\intd E) = \frac {\Gamma(\frac {n} 2) \Gamma(\frac {s} 2 + i + 1)}
      {\sqrt \pi \Gamma(\frac {n + s + 1} 2 + i)} \, Q^{\frac {s - 1}
        2 + i} \TenCM{n - 1}{r}{1}{0} (K, \beta).
    \end{align*}
  \end{theorem}

  Note that in Theorem~\ref{14-Thm_Main_k=1} the Crofton integral is
  expressed only by tensorial curvature measures $\TenCM{n -
    1}{r}{z}{0}$ (multiplied with suitable powers of the metric
  tensor), whereas generalized tensorial curvature measures are not
  needed. A global version of Theorem~\ref{14-Thm_Main_k=1} is
  obtained by simply setting $\beta = \R^n$.
			
  \medskip

  A translation invariant, global version of
  Theorem~\ref{14-Thm_Main_Gen_Loc} allows us to combine several of
  the summands on the right-hand side of the formula.

  \begin{theorem} \label{14-Thm_Main_Gen} Let $K \in \cK^n$ and $i, j,
    k, s \in \N_0$ with $j < k < n$ and $k > 1$. Then
    \begin{equation*}
      \int_{A(n, k)} Q(E)^i \, \IntMinTen{j}{0}{s}{E} (K \cap E) \,
      \mu_k(\intd E)  = \gamma_{n, k, j} \sum_{z = 0}^{\lfloor \frac
        {s} {2} \rfloor + i} \lambda_{n, k, j, s, i, z}^{(0)} \, Q^z
      \MinTen{n - k + j}{0}{s + 2i - 2z} (K),
    \end{equation*}
    where $\gamma_{n, k, j}$ and $\lambda_{n, k, j, s, i, z}^{(0)}$
    are defined as in Theorem~\ref{14-Thm_Main_Gen_Loc}, but
    \begin{align*}
      \MoveEqLeft[1] \vartheta_{n, k, j, s, i ,z , p, q}^{(0)} := (n - k + j)\bigl(\tfrac {k - 1} 2 + p\bigr) - \bigl( p(n -
      k) - q(k - 1) \bigr) \bigl( 1 + \tfrac {k - j - 1} {s + 2i - 2z
        - 1} (1 - \tfrac {z} {p + q})\bigr)
    \end{align*}
    replaces $\vartheta_{n, k, j, p, q}^{(0)}$, except if $s$ is odd
    and $z=\lfloor\frac {s} {2}\rfloor + i$, where $\lambda_{n, k, j,
      s, i, \lfloor\frac {s} {2}\rfloor + i}^{(0)}:=0$.
  \end{theorem}

  In Theorem~\ref{14-Thm_Main_Gen}, if $p = q = 0$, then the
  definition of $\lambda_{n, k, j, s, i, z}^{\smash{(0)}}$ implies
  that also $z=0$ and thus, $\vartheta_{n, k, j, s, i ,0 , 0,
    0}^{\smash{(0)}}$ is well-defined with $\tfrac {z} {p + q} = 1$.

  \subsection{Some Special Cases}\label{14-sec3.2}

  In the following, we restrict to the case $i=0$ of Crofton formulae
  for unweighted intrinsic Minkowski tensors or tensorial curvature
  measures.

  \begin{corollary} \label{14-Cor_Main_i=0} Let $K \in \cK^n$ and $k,
    j, s \in \N_0$ with $0 \leq j < k < n$. Then
    \begin{equation*}
      \int_{A(n, k)} \IntMinTen{j}{0}{s}{E} (K \cap E) \, \mu_k(\intd
      E) = \delta_{n, k, j, s} \sum_{z=0}^{\lfloor \frac s 2 \rfloor}
      \eta_{n, k, j, s, z} \, Q^z \MinTen{n - k + j}{0}{s - 2z} (K),
    \end{equation*}
    where
    \begin{align*}
      \delta_{n, k, j, s} : ={} & \binom{n - k + j - 1}{j}
      \frac{\Gamma(\frac {n - k + 1} 2) \Gamma(\frac {k + 1} 2)}{\pi
        \Gamma(\frac{n - k + j + s} 2 + 1)},
      \\
      \eta_{n, k, j, s, z} := {} & \sum_{q = z}^{\lfloor \frac s 2
        \rfloor} (-1)^{q - z} \binom{s}{2q} \binom{q}{z} \Gamma(q +
      \tfrac 1 2) \frac {\Gamma(\frac {j + s} 2 - q + 1) \Gamma(\frac
        {n - k} 2 + q)} {\Gamma(\frac {n + 1} 2 + q)}
      \\
      & \times \left( \tfrac {n - k + j} 2 + q + \tfrac {(k - j - 1)
          (q - z)} {s - 2z - 1} \right),
    \end{align*}
    but $\eta_{n, k, j, s, \lfloor \frac s 2 \rfloor} := 0$ if $s$ is
    odd.
  \end{corollary}

  \noindent
  \textbf{Specific choices of $s$}
  \\

  \noindent
  Next we collect some special cases of
  Corollary~\ref{14-Cor_Main_i=0}, which are obtained for specific
  choices of $s \in \N_{0}$ by applications of Legendre's duplication
  formula and elementary calculations.

        \begin{corollary}
          Let $K \in \cK^n$ and $k, j \in \N_0$ with $0 \leq j < k <
          n$. Then
          \begin{align*}
            \MoveEqLeft  \int_{A(n, k)} \IntMinTen{j}{0}{2}{E} (K \cap E) \, \mu_k(\intd E) \\
            & = \frac {\Gamma(\frac{k + 1} 2)\Gamma(\frac {n - k + j +
                1} 2)}{\Gamma(\frac {n + 3} 2)\Gamma(\frac {j + 1} 2)}
            \bigl( \tfrac {n - k} {4 (n - k + j)} \, Q \MinTen{n - k +
              j}{0}{0} (K) + \tfrac {n - k + nj + j} {2(n - k + j)} \,
            \MinTen{n - k + j}{0}{2} (K) \bigr).
          \end{align*}
        \end{corollary}

        \begin{corollary} \label{14-Cor_s=3} Let $K \in \cK^n$ and $k,
          j \in \N_0$ with $0 \leq j < k < n$. Then
          \begin{equation*}
            \int_{A(n, k)} \IntMinTen{j}{0}{3}{E} (K \cap E) \,
            \mu_k(\intd E)  = \frac {j + 1} {n - k + j + 1} \frac
            {\Gamma(\frac{k + 1} 2) \Gamma(\frac {n - k + j}
              2)}{\Gamma(\frac {n + 1} 2) \Gamma(\frac j 2)} \MinTen{n
              - k + j}{0}{3} (K).
          \end{equation*}
        \end{corollary}

        As $\Gamma(\frac j 2)^{-1} = 0$, for $j = 0$, the integral in
        Corollary~\ref{14-Cor_s=3} equals 0 in this case. However, as
        the integrand on the left-hand side is already 0, this is not
        surprising. The same is true for any odd number $s \in \N$ and
        $j = 0$.

        Corollary~\ref{14-Cor_s=3} immediately leads to a result which
        was obtained and applied by Bernig and Hug in
        \cite[Lemma~4.13]{14-BernHug15}.

        \begin{corollary}
          Let $K \in \cK^n$. Then
          \begin{equation*}
            \int_{A(n, 2)} \IntMinTen{1}{0}{3}{E} (K \cap E) \,
            \mu_k(\intd E) = \binom{n}{2}^{-1} \MinTen{n - 1}{0}{3}
            (K).
          \end{equation*}
        \end{corollary}
        $ $

        \noindent
        \textbf{The choice $j = k - 1$}
        \\

        \noindent
        Furthermore, we obtain simple Crofton formulae for the
        specific choice $j = k - 1$ in the local and in the global
        case.

        \begin{corollary} \label{14-Cor_j=k-1} Let $K \in \cK^n$,
          $\beta \in \cB(\R^n)$ and $k, r, s \in \N_0$ with $1 < k <
          n$. Then
          \begin{equation*}
            \int_{A(n, k)} \IntTenCM{k-1}{r}{s}{0}{E} (K \cap E, \beta
            \cap E) \, \mu_k(\intd E) = \delta_{n, k, k - 1, s}
            \sum_{z = 0}^{\lfloor \frac s 2 \rfloor} \xi_{n, k, s, z}
            Q^z \TenCM{n - 1}{r}{s-2z}{0} (K, \beta),
          \end{equation*}
          where
          \begin{equation*}
            \xi_{n, k, s, z} := \sum_{q = z}^{\lfloor \frac s 2
              \rfloor} (-1)^{q - z} \binom{s}{2q} \binom{q}{z} \Gamma
            \left(q + \tfrac 1 2 \right) \frac {\Gamma(\frac {k + s +
                1} 2 - q) \Gamma(\frac {n - k} 2 + q)}{\Gamma(\frac {n
                - 1 } 2 + q)}.
          \end{equation*}
        \end{corollary}

        Corollary~\ref{14-Cor_j=k-1} will be derived from
        Theorem~\ref{14-Thm_Main_Gen_Loc} in the same way as
        Theorem~\ref{14-Thm_Main_Gen} is proved. More specifically,
        relation \eqref{14-Form_Rel_tenCM} is applied, which can be
        considered as a local version of Lemma~\ref{14-Lem_McMullen}
        in the particular case $j = n - 1$.  Although $k=1$ is
        excluded in Corollary~\ref{14-Cor_j=k-1}, the result is
        formally consistent with Theorem~\ref{14-Thm_Main_k=1} (for
        $i=0$), which can be checked by simplifying the coefficients
        $\xi_{n, 1, s, z}$ with the help of Zeilberger's algorithm.
				
        A global version of Corollary~\ref{14-Cor_j=k-1} is obtained
        by setting $\beta = \R^n$.

        Finally, Theorem~\ref{14-Thm_Main_k=1} can be globalized to
        give a result, which was obtained in \cite{14-KousKideHug15}
        by a completely different approach.

        \begin{corollary} \label{14-Cor_k=1} Let $K \in \cK^n$ and $s
          \in \N_0$. Then
          \begin{equation*}
            \int_{A(n, 1)} \IntMinTen{0}{0}{s}{E} (K \cap E) \,
            \mu_k(\intd E) = \frac {2 \omega_{n + s + 1}} {\pi
              \omega_{s + 1} \omega_n} \sum_{z = 0}^{\frac s 2 } \frac
            {\left( -1 \right)^{z}}{1 - 2z} \binom{\frac s 2}{z}
            Q^{\frac s 2 - z} \MinTen{n - 1}{0}{2z} (K)
          \end{equation*}
          for even $s$. For odd $s$ the integral on the left-hand side
          equals 0.
        \end{corollary}

        Note that if $s \in \N$ is odd, then the Crofton integral in
        Theorem~\ref{14-Thm_Main_k=1} is a non-zero measure, as the
        tensorial curvature measures $\TenCM{n -
          1}{r}{1}{0}(K,\nularg)$ are non-zero (if the underlying set
        $K$ is at least $(n-1)$-dimensional), whereas $\MinTen{n -
          1}{0}{1} \equiv 0$ in the global case considered in
        Corollary~\ref{14-Cor_k=1}.

\subsection{Crofton Formulae for Extrinsic Tensorial Curvature
    Measures}\label{14-sec3.4}
		
		 In the following, we state  Crofton formulae for tensorial curvature measures for $j = k - 1$. The method
		also applies to the cases where $j\le k-2$, but it remains to be explored to which extent the constants can be simplified then.
		As for the intrinsic versions, we have to distinguish between the cases $k > 1$ and $k = 1$. We start with the former.

  \begin{theorem} \label{14-Thm_ExtrCF_j=k-1}
    Let $K \in \cK^{n}$, $\beta \in \cB(\R^{n})$ and $k, r, s \in \N_{0}$ with $1 < k < n$. Then
    \begin{align*}
      & \int_{A(n, k)} \TenCM{k - 1}{r}{s}{0} (K \cap E, \beta \cap E) \, \mu_k(\intd E) =
			\sum_{z = 0}^{\lfloor \frac s 2 \rfloor} \kappa_{n, k, s, z} \, Q^{z} \TenCM{n - 1}{r}{s - 2z}{0} (K, \beta),
    \end{align*}
    where
    $$
      \kappa_{n, k, s, z}: = \frac{k - 1} {n - 1} \frac {\pi^{\frac {n - k} 2} \Gamma(\frac {n} 2)} {\Gamma(\frac {k} 2) \Gamma(\frac {n - k} 2)} \frac{\Gamma(\frac {s + 1} 2) \Gamma(\frac s 2 + 1)}{\Gamma(\frac{n - k + s + 1}{2}) \Gamma(\frac {n + s - 1} 2)} \frac{\Gamma(\tfrac {n - k} 2 + z) \Gamma(\tfrac {k + s - 1} 2 - z)}{\Gamma(\frac s 2 - z + 1) z!}
    $$
    if $z \neq \frac {s - 1} 2$, and
    \begin{equation}\label{14-constkappa2}
      \kappa_{n, k, s, \frac {s - 1} 2} : = \pi^{\frac {n - k - 1} 2} \frac {2k (n + s - 2)} {(n - 1) (n - k + s - 1)} \frac {\Gamma(\frac {n} 2)}{\Gamma(\frac {n - k} 2)} \frac {\Gamma(\frac s 2 + 1)} { \Gamma(\frac {n + s + 1} 2) }.
    \end{equation}
  \end{theorem}

  In Theorem \ref{14-Thm_ExtrCF_j=k-1}, if $s$ is odd the coefficient $\kappa_{n, k, s,  ({s - 1})/ 2}$ has to be defined separately, as the proof shows. (In fact,  the difference amounts to a factor $ {k (n + s - 2)} [{(k - 1) (n + s - 1)}]^{-1}$.) For even  $s$, the  constants
	involved in the proof of Theorem \ref{14-Thm_ExtrCF_j=k-1} can be simplified by a direct calculation to arrive at the asserted result. However, if $s$ is odd, we need the connection to the work \cite{14-BernHug15} to simplify the
	constants. Since this connection breaks down for $z=(s-1)/2$, $s$ odd, a separate direct calculation is required for this case,
	and this finally yields the correct constant in \eqref{14-constkappa2}. The result is also consistent with the special case $k=1$ which is considered next.

  For $k = 1$ the Crofton integrals can be represented with a single functional, as the following theorem shows.

  \begin{theorem} \label{14-Thm_ExtrCF_k=1}
    Let $K \in \cK^{n}$, $\beta \in \cB(\R^{n})$ and $r, s \in \N_{0}$. Then
  \begin{align*}
    & \int_{A(n, 1)} \TenCM{0}{r}{s}{0} (K \cap E, \beta \cap E) \, \mu_1(\intd E) = \pi^{\frac {n - 2} 2} \frac {\Gamma(\frac {n} 2)} {\Gamma(\frac {n + 1} 2)} \frac {\Gamma(\lfloor \frac {s + 1} 2 \rfloor + \frac 1 2)} {\Gamma(\frac {n} 2 + \lfloor \frac {s + 1} 2 \rfloor)} Q^{\lfloor \frac s 2 \rfloor} \TenCM{n - 1}{r}{s - 2\lfloor \frac s 2 \rfloor}{0} (K, \beta).
  \end{align*}
  \end{theorem}
		
	It can be easily checked that the  result for $k=1$ can be obtained from the one for $k>1$ by a formal specialization and proper
		interpretation of expressions which a priori are not well defined. For this to work, it is indeed crucial that for odd values of $s$ and
		$z=(s-1)/2$ the definition in \eqref{14-constkappa2} applies.

		 In \cite[Proposition 4.10]{14-BernHug15},  an alternative basis of
		the vector space of continuous, translation invariant and rotation  covariant $\T^{p}$-valued valuations on $\cK^{n}$ was introduced,  based on the trace free part of the Minkowski tensors, which was called the $\Psi$-basis. In the same spirit (but locally and with the current normalization), we now define
  \begin{align*}
    \PTenCM{k}{r}{s}{0} := \TenCM{k}{r}{s}{0} + \frac 1 {\sqrt \pi} \sum_{j = 1}^{\lfloor \frac s 2 \rfloor} (-1)^{j} \binom{s}{2j} \frac{\Gamma(j + \frac 1 2) \Gamma(\frac n 2 + s - j - 1)}{\Gamma(\frac n 2 + s - 1)} Q^{j} \TenCM{k}{r}{s - 2j}{0}
  \end{align*}
	for $r,s\in \N_0$ and $k\in\{0,\ldots,n-1\}$.
  Interpreting this definition in the right way if $n=2$ and $s=0$ (where $\PTenCM{k}{r}{0}{0}=\TenCM{k}{r}{0}{0}$), we can also write
  \begin{align}
    \PTenCM{k}{r}{s}{0} = \frac 1 {\sqrt \pi} \sum_{j = 0}^{\lfloor \frac s 2 \rfloor} (-1)^{j} \binom{s}{2j} \frac{\Gamma(j + \frac 1 2) \Gamma(\frac n 2 + s - j - 1)}{\Gamma(\frac n 2 + s - 1)} Q^{j} \TenCM{k}{r}{s - 2j}{0}. \label{14-Def_Psi_Basis}
  \end{align}
In particular, $\PTenCM{k}{r}{s}{0}=\TenCM{k}{r}{s}{0}$ for $s\in\{0,1\}$. Conversely, we have
  \begin{align}
    \TenCM{k}{r}{s}{0} = \frac 1 {\sqrt \pi} \sum_{j = 0}^{\lfloor \frac s 2 \rfloor} \binom{s}{2j} \frac{\Gamma(j + \frac 1 2) \Gamma(\frac n 2 + s - 2j)}{\Gamma(\frac n 2 + s - j)} Q^{j} \PTenCM{k}{r}{s - 2j}{0}. \label{14-Form_Phi_Psi_Basis}
  \end{align}
Although this will not be needed explicit, it shows how we can switch between a $\phi$-representation and a $\psi$-representation of
tensorial curvature measures.

  The main advantage of the new local tensor valuations given in \eqref{14-Def_Psi_Basis} is that the Crofton formula takes a particularly simple form.
	
  \begin{corollary}\label{14-Cor_CF_Psi_Basis}
    Let $K \in \cK^{n}$, $\beta \in \cB(\R^{n})$, and let $k, r, s \in \N_{0}$ with $0 < k < n$. Then
    \begin{align*}
      & \int_{A(n, k)} \PTenCM{k - 1}{r}{s}{0} (K \cap E, \beta \cap E) \,\mu_{k}(\intd E) \\
      & \qquad = \pi^{\frac {n - k} 2} \frac{k - 1} {n - 1} \frac {\Gamma(\frac {n} 2) \Gamma(\frac{k + s - 1}{2})} {\Gamma(\frac {k} 2) \Gamma(\frac{n + s - 1}{2})} \frac{ \Gamma(\frac{s + 1}{2})}{\Gamma(\frac{n - k + s + 1}{2})} \,\PTenCM{n - 1}{r}{s}{0} (K, \beta).
    \end{align*}
  \end{corollary}
	
  For $r = 0$ and $\beta = \R^{n}$, Corollary \ref{14-Cor_CF_Psi_Basis} coincides with \cite[Corollary 6.1]{14-BernHug15} (in the case corresponding to $j=k-1$). If $s \in \{ 0, 1 \}$, then $\PTenCM{k}{r}{s}{0} = \TenCM{k}{r}{s}{0}$ and Corollary \ref{14-Cor_CF_Psi_Basis} coincides with Theorem~\ref{14-Thm_ExtrCF_j=k-1} (resp.~Theorem~\ref{14-Thm_ExtrCF_k=1}, for $k = 1$). If $k = 1$, then the integral in Corollary \ref{14-Cor_CF_Psi_Basis} vanishes, except for $s \in \{ 0, 1 \}$.

        \section{Proofs of the Main Results}\label{14-sec4}

        In this section, we first recall some results from
        \cite{14-HugSchnSchu08}. Then we prove an integral formula
        which is required in the following. Finally, all ingredients
        are combined for the proofs of our main theorems.

        A basic tool is the following transformation formula (see
        \cite[Corollary~4.2]{14-HugSchnSchu08}). It can be used to
        carry out an integration over linear Grassmann spaces
        recursively. The result is also true for $k=1$, but in this
        case the outer integration on the right-hand side is trivial.

        \begin{lemma} \label{14-Lem_Transf_Form} Let $u \in \Sn$ and
          let $h: G(n, k) \rightarrow \T^{p}$ be an integrable
          function for $k, p \in \N_{0}$, $0 < k < n$. Then
          \begin{align*}
            \MoveEqLeft \int _{G(n, k)} h(L) \, \nu_k (\intd L) =
            \frac {\omega_{k}} {2 \omega_{n}} \int _{G(u^\perp, k -
              1)} \int _{-1}^{1} \int _{U^{\perp} \cap u^{\perp} \cap
              \Sn} \abs{t}^{k - 1} ( 1 - t^2 )^{\frac {n - k - 2} 2}
            \\
            & \times h \bigl( \Span \bigl\{ U, t u + \sqrt{1 - t^2} w
            \bigr\} \bigr) \, \cH^{n - k - 1} (\intd w) \, \intd t \,
            \nu^{u^\perp}_{k - 1} (\intd U).
          \end{align*}
        \end{lemma}
			
        \medskip

        The next results are derived from the previous one (see
        \cite[Lemma 4.3 and Corollary 4.6]{14-HugSchnSchu08}).

        \begin{lemma} \label{14-Lem_Int_Form_1} Let $i, k \in \N_0$
          with $k \leq n$. Then
          \begin{equation*}
            \int _{G(n, k)} Q(L)^i \, \nu_{k} (\intd L) = \frac{\Gamma(\frac{n}{2}) \Gamma(\frac{k}{2} + i)}{\Gamma(\frac{n}{2} + i)\Gamma(\frac{k}{2})} Q^{i}.
          \end{equation*}
        \end{lemma}

        In Lemma~\ref{14-Lem_Int_Form_1}, we interpret the coefficient
        of the tensor on the right-hand side as 0, if $k = 0$ and $i
        \neq 0$, and as 1, if $k = i = 0$, as $\Gamma(0)^{-1} := 0$
        and $\frac{\Gamma(a)}{\Gamma(a)}=1$ for all $a \in \R$.

        \begin{lemma} \label{14-Lem_Int_Form_2} Let $i \in \N_0$, $k,r
          \in \{ 0, \ldots, n \}$ with $k + r \geq n$, and let $F \in
          G(n, r)$. Then
          \begin{align*}
            \int _{G(n, k)} [F, L]^{2} Q(L)^{i} \, \nu_{k} (\intd L) =
            {} & \frac {r! k!} {n! (k + r - n)!} \frac {\Gamma( \frac
              {n} 2 + 1) \Gamma( \frac k 2 + i)} {\Gamma( \frac n 2 +
              i + 1)\Gamma( \frac k 2 + 1)}
            \\
            & \times \bigl( ( \tfrac k 2 + i ) Q^i + i \tfrac {k - n}
            r Q^{i - 1} Q(F) \bigr).
          \end{align*}
        \end{lemma}

        We interpret the second summand on the right-hand side of
        Lemma~\ref{14-Lem_Int_Form_2} as 0, if $i = 0$, which is
        consistent with \cite[Lemma 4.4]{14-HugSchnSchu08}. If $r =
        0$, we also interpret the second summand as 0 and the integral
        on the left equals $Q^i$.

        Finally, we state the following integral formula (see
        \cite[p. 503]{14-HugSchnSchu08}), which is a special case of
        \cite[Theorem 3.1]{14-Rataj99}.

        \begin{lemma} \label{14-Croft_Form_Rataj} Let $P \in \cP^n$ be
          a polytope, $L \in G(n, k)$ for $0 \leq j < k < n$ and let
          $g: \R^n \times (\Sn \cap L) \rightarrow \T$ be a measurable
          bounded function. Then
          \begin{align*}
            \MoveEqLeft[3] \int _{L^\perp} \int_{L_t \times (L \cap
              \Sn)} g(x, u) \, \Lambda_j^{(L_t)}(P \cap L_t, \intd
            (x,u)) \, \cH^{n - k}(\intd t)
            \\
            = {} & \frac 1 {\omega_{k - j}} \sum_{F \in \cF_{n - k +
                j}(P)} \, \int_{F \times (N(P,F) \cap \Sn)} g(x,
            \pi_L(u)) \norm{p_L(u)}^{j - k} [ F, L ]^2 \, \cH^{n -
              1}(\intd (x, u)).
          \end{align*}
        \end{lemma}

        \subsection{Auxiliary Integral Formulae}\label{14-sec4.1}

        With the preliminary results from \cite{14-HugSchnSchu08} we
        are able to establish the following integral formula, which is
        a slightly modified version of
        \cite[Proposition~4.7]{14-HugSchnSchu08}.
        \begin{proposition} \label{14-Prop_Int_Form} Let $i, j, k, s
          \in \N_0$ with $j < k < n$ and $k > 1$, $F \in G(n, n - k +
          j)$ and $u \in F^{\perp} \cap \Sn$. Then
          \begin{align*}
            \MoveEqLeft \int _{G(n, k)} Q(L)^i \pi_L(u)^s \norm{p_L(u)
            }^{j - k} [ F, L ]^2 \, \nu_k(\intd L)
            \\
            & = \gamma_{n, k, j} \sum_{z = 0}^{\lfloor \frac s 2
              \rfloor + i} \bigl( \lambda_{n, k, j, s, i, z}^{(0)} u^2
            + \lambda_{n, k, j, s, i, z}^{(1)} Q(F) \bigr) Q^{z} u^{s
              + 2i - 2z - 2},
          \end{align*}
          where the coefficients are defined as in
          Theorem~\ref{14-Thm_Main_Gen_Loc}.
        \end{proposition}

    \begin{proof}
      Lemma~\ref{14-Lem_Transf_Form} yields
      \begin{align*}
        \MoveEqLeft[3] \int _{G(n, k)} Q(L)^i \pi_L(u)^s
        \norm{p_L(u)}^{j - k} [ F, L ]^2 \, \nu_k(\intd L)
        \\
        = {} & \frac {\omega_{k}} {2 \omega_{n}} \int _{G(u^\perp, k -
          1)} \int _{-1}^{1} \int _{U^{\perp} \cap u^{\perp} \cap \Sn}
        \abs{t}^{k - 1} ( 1 - t^2 )^{\frac {n - k - 2} 2} \pi_{\Span
          \{ U, t u + \sqrt{1 - t^2} w \}}(u)^s
        \\
        & \times Q \bigl( \Span \bigl\{ U, t u + \sqrt{1 - t^2} w
        \bigr\} \bigr)^i \norm{p_{\Span \{ U, t u + \sqrt{1 - t^2} w
            \}}(u)}^{j - k}
        \\
        & \times \bigl[ F, \Span \bigl\{ U, t u + \sqrt{1 - t^2} w
        \bigr\} \bigr]^2 \, \cH^{n - k - 1} (\intd w) \, \intd t \,
        \nu^{u^\perp}_{k - 1} (\intd U).
      \end{align*}
      As
      \begin{align*}
        Q \bigl( \Span \bigl\{ U, t u + \sqrt{1 - t^2} w \bigr\}
        \bigr) & = Q ( U ) + \bigl( \abs{t} u + \sqrt{1 - t^2} \sign
        (t) w \bigr)^2,
        \\
        \pi_{\Span \{ U, t u + \sqrt{1 - t^2} w \}}(u) & = \abs{t} u +
        \sqrt{1 - t^2} \sign(t) w,
        \\
        \norm{p_{\Span \{ U, t u + \sqrt{1 - t^2} w \}}(u)} & =
        \abs{t},
        \\
        \bigl[ F, \Span \bigl\{ U, t u + \sqrt{1 - t^2} w \bigr\}
        \bigr] & = [ F, U ]^{(u^\perp)} \abs{t}
      \end{align*}
      hold for all $t \in [-1, 1] \setminus \{ 0 \}$, we obtain
      \begin{align*}
        \MoveEqLeft[3] \int _{G(n, k)} Q(L)^i \pi_L(u)^s
        \norm{p_L(u)}^{j - k} [ F, L ]^2 \, \nu_k(\intd L)
        \\
        = {} & \frac {\omega_{k}} {2 \omega_{n}} \int _{G(u^\perp, k -
          1)} \int _{-1}^{1} \int _{U^{\perp} \cap u^{\perp} \cap \Sn}
        \abs{t}^{j + 1} ( 1 - t^2 )^{\frac {n - k - 2} 2} \bigl( [ F,
        U ]^{(u^\perp)} \bigr)^2 \bigl( \abs{t} u + \sqrt{1 - t^2} w \bigr)^s
        \\
        & \times \bigl( Q
        ( U ) + ( \abs{t} u + \sqrt{1 - t^2} w )^2 \bigr)^i  \, \cH^{n - k - 1} (\intd w) \, \intd t \,
        \nu^{u^\perp}_{k - 1} (\intd U),
      \end{align*}
      where we used the fact that the integration with respect to $w$
      is invariant under reflections in the origin.  Then we apply the
      binomial theorem to the terms $( Q ( U ) + ( |t| u + \sqrt{1 -
        t^2} w )^2 )^i$ and $( |t| u + \sqrt{1 - t^2} w )^{s + 2p}$
      and get
      \begin{align*}
        \MoveEqLeft[3] \int _{G(n, k)} Q(L)^i \pi_L(u)^s \norm{p_L(u)
        }^{j - k} [ F, L ]^2 \, \nu_k(\intd L)
        \\
        = {} & \frac {\omega_{k}} {2 \omega_{n}} \sum_{p = 0}^{i}
        \sum_{q = 0}^{s + 2p} \binom{i}{p} \binom{s + 2p}{q} \int
        _{G(u^\perp, k - 1)} \int _{-1}^{1} |t|^{j + s + 2p - q + 1} (
        1 - t^2 )^{\frac {n - k + q - 2} 2} \, \intd t
        \\
        & \times \int _{U^{\perp} \cap u^{\perp} \cap \Sn} w^{q} \,
        \cH^{n - k - 1} (\intd w) \bigl( [ F, U ]^{(u^\perp)} \bigr)^2
        u^{s + 2p - q} Q ( U )^{i - p} \, \nu^{u^\perp}_{k - 1} (\intd
        U).
      \end{align*}
      Since
      \begin{equation*}
        \int _{U^{\perp} \cap u^{\perp} \cap \Sn} w^{q} \, \cH^{n - k
          - 1} (\intd w) = \1 \{ q \text{ even}\} 2 \frac{\omega_{n -
            k + q}}{\omega_{q + 1}} Q(U^\perp \cap u^\perp)^{\frac q
          2},
      \end{equation*}
      we deduce from the definition of the Beta function and its
      relation to the Gamma function that
      \begin{align*}
        \MoveEqLeft[3] \int _{G(n, k)} Q(L)^i \pi_L(u)^s \norm{p_L(u)
        }^{j - k} [ F, L ]^2 \, \nu_k(\intd L)
        \\
        = {} & \frac {\omega_{k}} {\omega_{n}} \sum_{p = 0}^{i}
        \sum_{q = 0}^{\lfloor \frac s 2 \rfloor + p} \binom{i}{p}
        \binom{s + 2p}{2q} \frac {\Gamma(\frac {j + s} 2 + p - q + 1)
          \Gamma(\frac {n - k} 2 + q)} {\Gamma(\frac {n - k + j + s} 2
          + p + 1)} \frac {\omega_{n - k + 2q}} {\omega_{2q + 1}}
        \\
        & \times u^{s + 2p - 2q} \int _{G(u^\perp, k - 1)} Q (
        U^{\perp} \cap u^{\perp} )^q \bigl( [ F, U ]^{(u^\perp)}
        \bigr)^2 Q ( U )^{i - p} \, \nu^{u^\perp}_{k - 1} (\intd U).
      \end{align*}
      Applying the binomial theorem to $Q ( U^{\perp} \cap u^{\perp}
      )^q = (Q(u^{\perp}) - Q (U))^q$ yields
      \begin{align}
        \MoveEqLeft[3] \int _{G(n, k)} Q(L)^i \pi_L(u)^s
        \norm{p_L(u)}^{j - k} [ F, L ]^2 \, \nu_k(\intd L) \nonumber
        \\
        = {} & \frac {\Gamma(\frac {n} 2)} {\sqrt \pi \Gamma(\frac {k}
          2)} \sum_{p = 0}^{i} \sum_{q = 0}^{\lfloor \frac s 2 \rfloor
          + p} \sum_{y = 0}^{q} (-1)^{y} \binom{i}{p} \binom{s +
          2p}{2q} \binom{q}{y} \Gamma(q + \tfrac {1} 2) \frac {\Gamma(\frac {j + s} 2 + p - q + 1)}
        {\Gamma(\frac {n - k + j + s} 2 + p + 1)} \nonumber
        \\
        & \times u^{s + 2p - 2q} Q
        \bigl( u^{\perp} \bigr)^{q - y}  \int _{G(u^\perp, k - 1)} \bigl( [ F, U ]^{(u^\perp)}
        \bigr)^2 Q ( U )^{i - p + y} \, \nu^{u^\perp}_{k - 1} (\intd
        U). \label{14-Form_Prop_Int_Form_1}
      \end{align}
      We conclude from Lemma~\ref{14-Lem_Int_Form_2}, which is applied
      in $u^\perp$ to the remaining integral on the right-hand side of
      \eqref{14-Form_Prop_Int_Form_1},
      \begin{align*}
        \MoveEqLeft[3] \int _{G(n, k)} Q(L)^i \pi_L(u)^s
        \norm{p_L(u)}^{j - k} [ F, L ]^2 \, \nu_k(\intd L)
        \\
        = {} & \frac {(n - k + j)! (k - 1)!} {\sqrt \pi (n - 1)!  j!}
        \frac {\Gamma(\frac {n} 2) \Gamma( \frac {n + 1} 2)}
        {\Gamma(\frac {k} 2) \Gamma( \frac {k + 1} 2)} \sum_{p =
          0}^{i} \sum_{q = 0}^{\lfloor \frac s 2 \rfloor + p}
        \binom{i}{p} \binom{s + 2p}{2q} \Gamma(q + \tfrac {1} 2)
        \\
        & \times \frac {\Gamma(\frac {j + s} 2 + p - q + 1)}
        {\Gamma(\frac {n - k + j + s} 2 + p + 1)} u^{s + 2p - 2q}
        \sum_{y = 0}^{q} (-1)^{y} \binom{q}{y} \frac {\Gamma( \frac {k
            - 1} 2 + i - p + y)} {\Gamma( \frac {n + 1} 2 + i - p +
          y)}
        \\
        & \times \!
        \begin{aligned}[t]
          \Bigl( &\big( \tfrac {k - 1} 2 + i - p + y \big) Q \bigl(
          u^\perp \bigr)^{i - p + q} + \tfrac {k - n} {n - k + j} (i - p + y) Q \bigl( u^\perp
          \bigr)^{i - p + q - 1} Q(F) \Bigr).
        \end{aligned}
      \end{align*}
      Lemma~\ref{14-Lem_Zeil_1} from Section~\ref{14-sec5} applied
      twice to the summations with respect to $y$ and Legendre's
      duplication formula applied three times to the Gamma functions
      involving $n$, $k$ and $n - k$ yield together with the
      definitions of $\gamma_{n, k, j}$ and $\vartheta_{n, k, j, p,
        q}^{\smash{(\eps)}}$, $\eps \in \{0, 1\}$,
      \begin{align*}
        \MoveEqLeft[3] \int _{G(n, k)} Q(L)^i \pi_L(u)^s
        \norm{p_L(u)}^{j - k} [ F, L ]^2 \, \nu_k(\intd L)
        \\
        = {} & \gamma_{n, k, j} \sum_{p = 0}^{i} \sum_{q = 0}^{\lfloor
          \frac s 2 \rfloor + i - p} \binom{i}{p} \binom{s + 2i -
          2p}{2q} \Gamma(q + \tfrac {1} 2)
        \\
        & \times \frac {\Gamma(\frac {j + s} 2 + i - p - q + 1)}
        {\Gamma(\frac {n - k + j + s} 2 + i - p + 1)} \frac
        {\Gamma(\frac {k - 1} 2 + p) \Gamma(\frac {n - k} 2 + q)}
        {\Gamma(\frac {n + 1} 2 + p + q)}
        \\
        & \times u^{s + 2i - 2p - 2q} \Bigl( \vartheta_{n, k, j, p,
          q}^{(0)} Q \bigl( u^\perp \bigr)^{p + q} - \vartheta_{n, k,
          j, p, q}^{(1)} Q \bigl( u^\perp \bigr)^{p + q - 1} Q(F)
        \Bigr),
      \end{align*}
      where we changed the order of summation with respect to $p$.
      From the binomial theorem applied to $Q(u^\perp)^{p + q} = (Q -
      u^2)^{p + q}$ we obtain
      \begin{align*}
        \MoveEqLeft[2] \int _{G(n, k)} Q(L)^i \pi_L(u)^s
        \norm{p_L(u)}^{j - k} [ F, L ]^2 \, \nu_k(\intd L)
        \\
        = {} & \gamma_{n, k, j} \sum_{p = 0}^{i} \sum_{q = 0}^{\lfloor
          \frac s 2 \rfloor + i - p} \binom{i}{p} \binom{s + 2i -
          2p}{2q} \Gamma(q + \tfrac {1} 2) \frac {\Gamma(\frac {j + s}
          2 + i - p - q + 1)} {\Gamma(\frac {n - k + j + s} 2 + i - p
          + 1)}
        \\
        & \times \frac {\Gamma(\frac {k - 1} 2 + p) \Gamma(\frac {n -
            k} 2 + q)} {\Gamma(\frac {n + 1} 2 + p + q)} \biggl(
        \sum_{z = 0}^{p + q} (-1)^{p + q - z} \binom{p + q}{z}
        \vartheta_{n, k, j, p, q}^{(0)} Q^{z} u^{s + 2i - 2z}
        \\
        & + \sum_{z = 0}^{p + q - 1} (-1)^{p + q - z } \binom{p + q -
          1}{z} \vartheta_{n, k, j, p, q}^{(1)} Q^z u^{s + 2i - 2z -
          2} Q(F) \biggr).
      \end{align*}
      A change of the order of summation, such that we sum with
      respect to $z$ first, gives
      \begin{align*}
        \MoveEqLeft \int _{G(n, k)} Q(L)^i \pi_L(u)^s \norm{p_L(u)}^{j
          - k} [ F, L ]^2 \, \nu_k(\intd L)
        \\
        & = \gamma_{n, k , j} \sum_{z = 0}^{\lfloor \frac s 2 \rfloor
          + i} \bigl( \lambda_{n, k, j, s, i, z}^{(0)} u^2 +
        \lambda_{n, k, j, s, i, z}^{(1)} Q(F) \bigr) Q^{z} u^{s + 2i -
          2z - 2},
      \end{align*}
      which concludes the proof.
    \end{proof}

    Next we state the special case of
    Proposition~\ref{14-Prop_Int_Form} where $k = 1$.

    \begin{proposition} \label{14-Prop_Int_Form_k=1} Let $i, s \in
      \N_0$, $F \in G(n, n - 1)$ and $u \in F^{\perp} \cap \Sn$. Then
      \begin{align*}
        \MoveEqLeft \int _{G(n, 1)} Q(L)^i \pi_L(u)^s
        \norm{p_L(u)}^{-1} [ F, L ]^2 \, \nu_1(\intd L)
        \\
        & = \frac {\Gamma(\frac {n} 2) \Gamma(\frac {s + 1} 2 + i)}
        {\pi \Gamma(\frac {n + s + 1} 2 + i)} \sum_{z = 0}^{\frac s 2
          + i} (-1)^{z} \binom{\frac s 2 + i}{z} \frac{1}{1 - 2z}
        u^{2z} Q^{\frac s 2 + i - z}
      \end{align*}
      for even $s$. If $s$ is odd, then
      \begin{equation*}
        \int _{G(n, 1)} Q(L)^i \pi_L(u)^s \norm{p_L(u)}^{-1} [ F, L
        ]^2 \, \nu_k(\intd L) = \frac {\Gamma(\frac {n} 2)
          \Gamma(\frac s 2 + i + 1)} {\sqrt \pi \Gamma(\frac {n + s +
            1} 2 + i)}  u Q^{\frac {s - 1} 2 + i}.
      \end{equation*}
    \end{proposition}

    \begin{proof}
      The proof basically works as the proof of
      Proposition~\ref{14-Prop_Int_Form}. But we do not need to apply
      Lemma~\ref{14-Lem_Int_Form_2} as \eqref{14-Form_Prop_Int_Form_1}
      simplifies to
      \begin{align*}
        \MoveEqLeft[3] \int _{G(n, 1)} Q(L)^i \pi_L(u)^s
        \norm{p_L(u)}^{-1} [ F, L ]^2 \, \nu_k(\intd L)
        \\
        = {} & \frac {\Gamma(\frac {n} 2)} {\pi} \sum_{p = 0}^{i}
        \sum_{q = 0}^{\lfloor \frac s 2 \rfloor + p} \sum_{y = 0}^{q}
        (-1)^{y} \binom{i}{p} \binom{s + 2p}{2q} \binom{q}{y} \Gamma(q
        + \tfrac {1} 2) \frac {\Gamma(\frac {s} 2 + p - q + 1)}
        {\Gamma(\frac {n + s + 1} 2 + p)}
        \\
        & \times u^{s + 2p - 2q} Q \bigl( u^{\perp} \bigr)^{q - y}
        \int _{G(u^\perp, 0)} \bigl( [ F, U ]^{(u^\perp)} \bigr)^2 Q (
        U )^{i - p + y} \, \nu^{u^\perp}_{k - 1} (\intd U).
      \end{align*}
      Since the remaining integral on the right-hand side equals 1, if
      $p = i$ and $y = 0$, and in all the other cases it equals 0, we
      obtain
      \begin{align*}
        \MoveEqLeft \int _{G(n, k)} Q(L)^i \pi_L(u)^s \norm{p_L(u)}^{j
          - k} [ F, L ]^2 \, \nu_k(\intd L)
        \\
        & = \frac {\Gamma(\frac {n} 2)} {\pi} \sum_{q = 0}^{\lfloor
          \frac s 2 \rfloor + i} \binom{s + 2i}{2q} \Gamma(q + \tfrac
        {1} 2) \frac {\Gamma(\frac {s} 2 + i - q + 1)} {\Gamma(\frac
          {n + s + 1} 2 + i)} u^{s + 2i - 2q} Q \bigl( u^{\perp}
        \bigr)^{q}.
      \end{align*}
      Applying the binomial theorem to $Q ( u^{\perp} )^{q} = (Q -
      u^2)^{q}$ yields
      \begin{align*}
        \MoveEqLeft[1] \int _{G(n, k)} Q(L)^i \pi_L(u)^s
        \norm{p_L(u)}^{j - k} [ F, L ]^2 \, \nu_k(\intd L)
        \\
        & = \frac {\Gamma(\frac {n} 2)} {\pi} \sum_{q = 0}^{\lfloor
          \frac s 2 \rfloor + i} \sum_{z = 0}^{q} (-1)^{q - z}
        \binom{s + 2i}{2q} \binom{q}{z} \Gamma(q + \tfrac {1} 2) \frac
        {\Gamma(\frac {s} 2 + i - q + 1)} {\Gamma(\frac {n + s + 1} 2
          + i)} u^{s + 2i - 2z} Q^{z}.
      \end{align*}
      A change of the order of summation and Legendre's duplication
      formula applied to the Gamma functions involving $q$ give
      \begin{align*}
        \MoveEqLeft \int _{G(n, k)} Q(L)^i \pi_L(u)^s \norm{p_L(u)}^{j
          - k} [ F, L ]^2 \, \nu_k(\intd L)
        \\
        & = \frac {(s + 2i)! \Gamma(\frac {n} 2)} {2^{s + 2i}
          \Gamma(\frac {n + s + 1} 2 + i)} \sum_{z = 0}^{\lfloor \frac
          s 2 \rfloor + i} \frac{1}{z!} \sum_{q = z}^{\lfloor \frac s
          2 \rfloor + i} \frac{(-1)^{q - z} }{\Gamma(\frac {s + 1} 2 +
          i - q)(q - z)!} u^{s + 2i - 2z} Q^{z}.
      \end{align*}
      If $s$ is even, we conclude from Lemma~\ref{14-Lem_Zeil_2}
      applied to the summation with respect to~ $q$ and from another
      application of Legendre's duplication formula that
      \begin{align*}
        \MoveEqLeft \int _{G(n, k)} Q(L)^i \pi_L(u)^s \norm{p_L(u)}^{j
          - k} [ F, L ]^2 \, \nu_k(\intd L)
        \\
        & = \frac {\Gamma(\frac {n} 2) \Gamma(\frac {s + 1} 2 + i)}
        {\pi \Gamma(\frac {n + s + 1} 2 + i)} \sum_{z = 0}^{\frac s 2
          + i} (-1)^{\frac s 2 + i - z + 1} \binom{\frac s 2 + i}{z}
        \frac{1}{s + 2i - 2z - 1} u^{s + 2i - 2z} Q^{z}.
      \end{align*}
      A change of the order of summation with respect to $z$ then
      yields the assertion.

      On the other hand, if $s$ is odd, the binomial theorem gives,
      for $\lfloor \frac s 2 \rfloor + i \neq z$,
      \begin{align}
        \sum_{q = z}^{\lfloor \frac s 2 \rfloor + i} \frac{(-1)^{q -
            z} }{\Gamma(\frac {s + 1} 2 + i - q)(q - z)!} & = \frac 1
        {(\lfloor \frac s 2 \rfloor + i - z)!} \sum_{q = 0}^{\lfloor
          \frac s 2 \rfloor + i - z} (-1)^{q} \binom{\lfloor \frac s 2
          \rfloor + i - z}{q} \nonumber
        \\
        & = \frac 1 {(\lfloor \frac s 2 \rfloor + i - z)!} (1 -
        1)^{\lfloor \frac s 2 \rfloor + i - z} \nonumber
        \\
        & = 0. \label{14-Form_Prop_Int_Form_k=1}
      \end{align}
      For $\lfloor \frac s 2 \rfloor + i = z$, the sum on the
      left-hand side of \eqref{14-Form_Prop_Int_Form_k=1} equals 1.
      Hence, we finally obtain
      \begin{equation*}
        \int _{G(n, k)} Q(L)^i \pi_L(u)^s \norm{p_L(u)}^{j - k} [ F, L
        ]^2 \, \nu_k(\intd L) = \frac {\Gamma(\frac {n} 2)
          \Gamma(\frac s 2 + i + 1)} {\sqrt \pi \Gamma(\frac {n + s +
            1} 2 + i)}  u Q^{\lfloor \frac s 2 \rfloor + i},
      \end{equation*}
      if $s$ is odd.
    \end{proof}

    \subsection{The Proofs for the Intrinsic Case}\label{14-sec4.2}

    Now all tools are available which are needed to prove the main
    theorems.
		
    We start with the proof of
    Theorem~\ref{14-Thm_Main_j=k}.

    \begin{proof}[Proof (Theorem~\ref{14-Thm_Main_j=k})]
      Let $L \in G(n, k)$ and $t \in L^\perp$. Then we have
      \begin{equation*}
        \IntTenCM{k}{r}{s}{0}{L_t}(K \cap L_t, \beta \cap L_t) = \1 \{
        s = 0 \} \int _{K \cap \beta \cap L_t} x^r \, \cH^{k}(\intd x)
      \end{equation*}
      and thus, for $s \neq 0$,
      \begin{align*}
        \MoveEqLeft \int_{A(n, k)} Q(E)^i \IntTenCM{k}{r}{s}{0}{E} (K
        \cap E, \beta \cap E) \, \mu_k(\intd E)
        \\
        & = \int_{G(n, k)} \int _{L^\perp} Q(L_t)^i
        \IntTenCM{k}{r}{s}{0}{L_t} (K \cap L_t, \beta \cap L_t) \,
        \cH^{n - k} (\intd t) \, \nu_k(\intd L) = 0.
      \end{align*}
      Furthermore, for $s = 0$ Fubini's theorem yields
      \begin{align*}
        \MoveEqLeft \int _{A(n, k)} Q(E)^i \IntTenCM{k}{r}{0}{0}{E} (K
        \cap E, \beta \cap E) \, \mu_k(\intd E)
        \\
        & = \int_{G(n, k)} Q(L)^i \int _{L^\perp} \int _{K \cap \beta
          \cap L_t} x^r \, \cH^{k}(\intd x) \, \cH^{n - k} (\intd t)
        \, \nu_k(\intd L)
        \\
        & = \int_{G(n, k)} Q(L)^i \, \nu_k(\intd L) \int _{K \cap
          \beta} x^r \, \cH^{n}(\intd x).
      \end{align*}
      Then we conclude the proof with Lemma~\ref{14-Lem_Int_Form_1}
      and the definition of $\TenCM{n}{r}{0}{0}$.
    \end{proof}
		

    We turn to the proof of Theorem~\ref{14-Thm_Main_Gen_Loc}.

    \begin{proof}[Proof (Theorem~\ref{14-Thm_Main_Gen_Loc})]
      First, we prove the formula for a polytope $P \in \cP^n$. The
      general result then follows by an approximation argument.

      As a matter of convenience, we name the integral of interest
      $I$. Then Lemma~\ref{14-Croft_Form_Rataj} yields
      \begin{align*}
        I & = \omega_{k - j} \int_{G(n, k)} Q(L)^i \int _{L^\perp}
        \int _{L_t \times (L \cap \Sn)} \1_{\beta} (x) x^r u^s  \Lambda_j^{(L_t)}(P \cap L_t, \intd
        (x, u)) \, \cH^{n - k} (\intd t) \, \nu_k(\intd L)
        \\
        & = \sum_{F \in \cF_{n - k + j}(P)} \int_{F \cap \beta} x^r
        \cH^{n - k + j} (\intd x) \int_{G(n, k)} Q(L)^i
        \\
        & \qquad\qquad\qquad \times \int_{N(P,F) \cap \Sn} \pi_L(u)^s
        \| p_L(u) \|^{j - k} [ F, L ]^2 \, \cH^{k - j - 1}(\intd u) \,
        \nu_k(\intd L).
      \end{align*}
      With Fubini's theorem we conclude
      \begin{align}
        I = {} & \sum_{F \in \cF_{n - k + j}(P)} \int_{F \cap \beta}
        x^r \cH^{n - k + j} (\intd x) \int_{N(P,F) \cap \Sn} \nonumber
        \\
        & \times \int_{G(n, k)} Q(L)^i \pi_L(u)^s \norm{p_L(u) }^{j -
          k} [ F, L ]^2 \, \nu_k(\intd L) \, \cH^{k - j - 1}(\intd
        u). \label{14-Form_Thm_Main_1}
      \end{align}
      Then we obtain from Proposition~\ref{14-Prop_Int_Form}
      \begin{align*}
        \MoveEqLeft[3] I = \gamma_{n, k, j} \sum_{F \in \cF_{n - k +
            j}(P)} \int_{F \cap \beta} x^r \cH^{n - k + j} (\intd x)
        \\
        & \times \!
        \begin{aligned}[t]
          \biggl(& \sum_{z = 0}^{\lfloor \frac s 2 \rfloor + i}
          \lambda_{n, k, j, s, i, z}^{(0)} Q^z \int_{N(P,F) \cap \Sn}
          u^{s + 2i - 2z} \, \cH^{k - j - 1}(\intd u)
          \\
          & + \sum_{z = 0}^{\lfloor \frac s 2 \rfloor + i - 1}
          \lambda_{n, k, j, s, i, z}^{(1)} Q^z Q(F) \int_{N(P,F) \cap
            \Sn} u^{s + 2i - 2z - 2} \, \cH^{k - j - 1}(\intd u)
          \biggr).
        \end{aligned}
      \end{align*}
      With the definition of the tensorial curvature measures we get
      \begin{align*}
        I = {} & \gamma_{n, k, j} \sum_{z = 0}^{\lfloor \frac s 2
          \rfloor + i} \lambda_{n, k, j, s, i, z}^{(0)} Q^z \TenCM{n -
          k + j}{r}{s + 2i - 2z}{0} (P, \beta) + \gamma_{n, k, j} \sum_{z = 0}^{\lfloor \frac s 2 \rfloor +
          i - 1} \lambda_{n, k, j, s, i, z}^{(1)} Q^z \TenCM{n - k +
          j}{r}{s + 2i - 2z - 2}{1} (P, \beta).
      \end{align*}
      Combining the two sums yields the assertion in the polytopal
      case.

      As pointed out before, there exists a weakly continuous extension
      of the generalized tensorial curvature measures $\smash{\TenCM{n
          - k + j}{r}{s + 2i - 2z - 2}{1}}$ from the set of all polytopes to $\cK^{n}$.
      The same is true for the tensorial curvature measures $\TenCM{n
        - k + j}{r}{s + 2i - 2z}{0}$. Hence, approximating a convex
      body $K \in \cK^{n}$ by polytopes yields the assertion in the
      general case.
    \end{proof}

    Now we prove Theorem~\ref{14-Thm_Main_k=1}, which deals with the
    case $k=1$ excluded in the statement of
    Theorem~\ref{14-Thm_Main_Gen_Loc}.

    \begin{proof}[Proof (Theorem~\ref{14-Thm_Main_k=1})]
      The proof basically works as the one of
      Theorem~\ref{14-Thm_Main_Gen_Loc}. Again, we prove the formula
      for a polytope $P \in \cP^n$. We call the integral of interest
      $I$ and proceed as in the previous proof in order to obtain
      \eqref{14-Form_Thm_Main_1}. Now we apply
      Proposition~\ref{14-Prop_Int_Form_k=1} and obtain
      \begin{align*}
        I = {} & \frac {\Gamma(\frac {n} 2) \Gamma(\frac {s + 1} 2 +
          i)} {\pi \Gamma(\frac {n + s + 1} 2 + i)} \sum_{z =
          0}^{\frac s 2 + i} (-1)^{z} \binom{\frac s 2 + i}{z}
        \frac{1}{1 - 2z}
        Q^{\frac s 2 + i - z} \\
        & \times \sum_{F \in \cF_{n - 1}(P)} \int_{F \cap \beta} x^r
        \cH^{n - k + j} (\intd x) \int_{N(P,F) \cap \Sn} u^{2z} \,
        \cH^{0}(\intd u),
      \end{align*}
      if $s$ is even. Hence, we conclude the assertion with the
      definition of $\TenCM{n - 1}{r}{2z}{0}$.

      If $s$ is odd, Proposition~\ref{14-Prop_Int_Form_k=1} yields
      \begin{equation*}
        I = \frac {\Gamma(\frac {n} 2) \Gamma(\frac {s} 2 + i + 1)} {\sqrt
				\pi \Gamma(\frac {n + s + 1} 2 + i)} Q^{\frac {s - 1} 2 + i}
				\TenCM{n - 1}{r}{1}{0} (P, \beta).
      \end{equation*}
      As sketched in the proof of Theorem~\ref{14-Thm_Main_Gen_Loc},
      the general result follows by an approximation argument.
    \end{proof}

    For the proof of Theorem~\ref{14-Thm_Main_Gen}, we first globalize
    Theorem~\ref{14-Thm_Main_Gen_Loc} and then apply
    Lemma~\ref{14-Lem_McMullen} to treat the appearing tensors
    $\smash{\TenCM{n - k + j}{0}{s + 2i - 2z - 2}{1}}$.

    \begin{proof}[Proof (Theorem~\ref{14-Thm_Main_Gen})]
      We only prove the formula for a polytope $P \in \cP^n$. As
      before, the general result follows by an approximation argument.

      We briefly write $I$ for the Crofton integral under
      investigation. Starting from the special case of
      Theorem~\ref{14-Thm_Main_Gen_Loc} where $r = 0$ and
      $\beta=\R^n$, we obtain
      \begin{align*}
        I = {} & \gamma_{n, k, j} \sum_{z = 0}^{\lfloor \frac s 2
          \rfloor + i} \lambda_{n, k, j, s, i, z}^{(0)} Q^z \MinTen{n
          - k + j}{0}{s + 2i - 2z} (P) + \gamma_{n, k, j} \sum_{z =
          0}^{\lfloor \frac s 2 \rfloor + i - 1} \lambda_{n, k, j, s,
          i, z}^{(1)} Q^z
        \\
        & \qquad \times \sum_{F \in \cF_{n - k + j}(P)} Q(F) \cH^{n -
          k + j} (F) \int _{N(P,F) \cap \Sn} u^{s + 2i - 2z - 2} \,
        \cH^{k - j - 1} (\intd u).
      \end{align*}
      With $Q(F) = Q - Q(N(P, F))$ and Lemma~\ref{14-Lem_McMullen} we
      get
      \begin{align*}
        \MoveEqLeft \sum_{F \in \cF_{n - k + j}(P)} Q(F) \cH^{n - k +
          j}(F) \int _{N(P,F) \cap \Sn} u^{s + 2i - 2z - 2} \, \cH^{k
          - j - 1} (\intd u)
        \\
        & = Q \MinTen{n - k + j}{0}{s + 2i - 2z - 2} (P) - \tfrac {k -
          j + s + 2i - 2z - 2} {s + 2i - 2z - 1} \MinTen{n - k +
          j}{0}{s + 2i - 2z} (P)
      \end{align*}
      and thus
      \begin{align*}
        I = {} & \gamma_{n, k, j} \sum_{z = 0}^{\lfloor \frac s 2
          \rfloor + i} \lambda_{n, k, j, s, i, z}^{(0)} Q^z \MinTen{n
          - k + j}{0}{s + 2i - 2z} (P)
        \\
        & + \gamma_{n, k, j} \sum_{z = 0}^{\lfloor \frac s 2 \rfloor +
          i - 1} \lambda_{n, k, j, s, i, z}^{(1)} Q^z \bigl( Q
        \MinTen{n - k + j}{0}{s + 2i - 2z - 2} (P) - \tfrac {k - j + s
          + 2i - 2z - 2} {s + 2i - 2z - 1} \MinTen{n - k + j}{0}{s +
          2i - 2z} (P) \bigr).
      \end{align*}
      Combining these sums yields
      \begin{equation*}
        I  = \gamma_{n, k, j} \sum_{z = 0}^{\lfloor \frac s 2 \rfloor
          + i} \bigl( \lambda_{n, k, j, s, i, z}^{(0)} + \lambda_{n,
          k, j, s, i, z - 1}^{(1)} - \tfrac {k - j + s + 2i - 2z - 2}
        {s + 2i - 2z - 1} \lambda_{n, k, j, s, i, z}^{(1)} \bigr) Q^z
        \MinTen{n - k + j}{0}{s + 2i - 2z} (P).
      \end{equation*}
      In fact, we have $\lambda_{n, k, j, s, i, -1}^{\smash{(1)}} = 0$
      and, furthermore for even $s$, as the sum with respect to $q$ is
      empty, $\lambda_{n, k, j, s, i, \lfloor \frac s 2 \rfloor +
        i}^{\smash{(1)}}$ also vanishes. On the other hand, for odd
      $s$, as $\MinTen{n - k + j}{0}{1} \equiv 0$, the last summand of
      the sum with respect to $z$ actually vanishes and thus its
      coefficient does not have to be determined and is defined as
      zero.

      Hence, we obtained a representation of the integral with the
      desired Min\-kowski tensors. It remains to determine the
      coefficients explicitly. First, we consider the case where ($k >
      1$ and) $z \in \{ 1, \ldots, \lfloor \frac s 2 \rfloor + i - 1
      \}$. We get
      \begin{align*}
        \MoveEqLeft[3] \lambda_{n, k, j, s, i, z}^{(0)} + \lambda_{n,
          k, j, s, i, z - 1}^{(1)}
        \\
        = {} & \sum_{p = 0}^{i} \sum_{q = (z - p)^+}^{\lfloor \frac s
          2 \rfloor + i - p} (-1)^{p + q - z} \binom{i}{p} \binom{s +
          2i - 2p}{2q} \binom{p + q}{z} \Gamma(q + \tfrac 1 2)
        \\
        & \times \frac {\Gamma(\frac {j + s} 2 + i - p - q + 1)
          \Gamma(\frac {k - 1} 2 + p) \Gamma(\frac {n - k} 2 + q)}
        {\Gamma(\frac {n - k + j + s} 2 + i - p + 1)\Gamma(\frac {n +
            1} 2 + p + q) }
        \\
        & \times \Bigl( (n- k + j)(\tfrac {k - 1} 2 + p) - \tfrac {z}
        {p + q} \bigl(p(n - k) - q(k - 1)\bigr) \Bigr)
      \end{align*}
      and
      \begin{align}
        \lambda_{n, k, j, s, i, z}^{(1)} = {} & \sum_{p = 0}^{i}
        \sum_{q = (z - p)^+}^{\lfloor \frac s 2 \rfloor + i - p}
        (-1)^{p + q - z} \binom{i}{p} \binom{s + 2i - 2p}{2q} \binom{p
          + q}{z} \Gamma(q + \tfrac 1 2) \nonumber
        \\
        & \times \frac {\Gamma(\frac {j + s} 2 + i - p - q + 1)
          \Gamma(\frac {k - 1} 2 + p) \Gamma(\frac {n - k} 2 + q)}
        {\Gamma(\frac {n - k + j + s} 2 + i - p + 1)\Gamma(\frac {n +
            1} 2 + p + q)} \nonumber
        \\
        & \times \tfrac {p + q - z} {p + q} \bigl(p(n - k) - q(k -
        1)\bigr). \label{14-Form_Proof_Thm_Main_Gen}
      \end{align}
      Hence we conclude
      \begin{align*}
        \MoveEqLeft[3] \lambda_{n, k, j, s, i, z}^{(0)} + \lambda_{n,
          k, j, s, i, z - 1}^{(1)} - \tfrac {k - j + s + 2i - 2z - 2}
        {s + 2i - 2z - 1} \lambda_{n, k, j, s, i, z}^{(1)}
        \\
        = {} & \sum_{p = 0}^{i} \sum_{q = (z - p)^+}^{\lfloor \frac s
          2 \rfloor + i - p} (-1)^{p + q - z} \binom{i}{p} \binom{s +
          2i - 2p}{2q} \binom{p + q}{z} \Gamma(q + \tfrac 1 2)
        \\
        & \times \frac {\Gamma(\frac {j + s} 2 + i - p - q + 1)
          \Gamma(\frac {k - 1} 2 + p) \Gamma(\frac {n - k} 2 + q)}
        {\Gamma(\frac {n - k + j + s} 2 + i - p + 1)\Gamma(\frac {n +
            1} 2 + p + q) }
        \\
        & \times \left( (n - k + j) (\tfrac {k - 1} 2 + p) - \tfrac
          {p(n - k) - q(k - 1)} {p + q} \left( p + q + \tfrac {(k - j
              - 1)(p + q - z)} {s + 2i - 2z - 1} \right) \right).
      \end{align*}
      The case $z = \lfloor \frac s 2 \rfloor + i$, for even $s$,
      follows similarly.  For $z = 0$, we have $\lambda_{n, k, j, s,
        i, -1}^{(1)} = 0$ and \eqref{14-Form_Proof_Thm_Main_Gen} still
      holds, if one cancels the remaining $\frac {p + q - z} {p + q} =
      1$.
    \end{proof}

    Finally, we provide the argument for Corollary~\ref{14-Cor_j=k-1},
    which is the special case of Theorem~\ref{14-Thm_Main_Gen_Loc}
    obtained for $i=0$ and $j+1=k\ge 2$.

    \begin{proof}[Proof (Corollary~\ref{14-Cor_j=k-1})] With the specific
      choices of the indices, we obtain
      \begin{align*}
        \lambda^{(\varepsilon)}_{n,k,k-1,s,0,z}= {} & \sum_{q = z +
          \varepsilon}^{\lfloor \frac s 2 \rfloor} (-1)^{q - z}
        \binom{s }{2q} \binom{ q - \varepsilon}{z} \Gamma(q + \tfrac 1
        2)
        \\
        & \times \frac {\Gamma(\frac {k + s + 1} 2 - q )}
        {\Gamma(\frac {n + s + 1} 2 )} \frac {\Gamma(\frac {k - 1} 2 )
          \Gamma(\frac {n - k} 2 + q)} {\Gamma(\frac {n + 1} 2 + q)}
        \vartheta_{n, k, k-1, 0, q}^{(\varepsilon)},
      \end{align*}
      with
      \begin{equation*}
        \vartheta_{n, k, k-1, 0, q}^{(0)} = \tfrac{1}{2}(n -1) (k-1), \qquad
        \vartheta_{n, k, k-1, 0, q}^{(1)} := - q(k - 1),
      \end{equation*}
      and
      \begin{equation*}
        \gamma_{n, k, k-1} = \binom{n - 2}{k-1} \frac{\Gamma(\frac {n - k + 1}
          2)}{2 \pi}.
      \end{equation*}
      Let us denote the Crofton integral by $I$. Then
      Theorem~\ref{14-Thm_Main_Gen_Loc} implies that
      \begin{align*}
        I= {} &\gamma_{n,k,k-1}\sum_{z=0}^{\lfloor \frac s 2
          \rfloor}Q^z\bigl(\lambda^{(0)}_{n,k,k-1,s,0,z}-\lambda^{(1)}_{n,k,k-1,s,0,z}\bigr)
        \TenCM{n -1}{r}{s - 2z }{0}(K,\beta)
        \\
        & + \gamma_{n,k,k-1}\sum_{z=1}^{\lfloor \frac s 2
          \rfloor+1}Q^z \lambda^{(1)}_{n,k,k-1,s,0,z-1} \TenCM{n
          -1}{r}{s - 2z }{0}(K,\beta)
        \\
        = {} &\gamma_{n,k,k-1}\sum_{z=0}^{\lfloor \frac s 2
          \rfloor}Q^z\bigl(\underbrace{\lambda^{(0)}_{n,k,k-1,s,0,z}+\lambda^{(1)}_{n,k,k-1,s,0,z-1}-
          \lambda^{(1)}_{n,k,k-1,s,0,z}}_{=:\lambda}\bigr) \TenCM{n
          -1}{r}{s - 2z }{0}(K,\beta),
      \end{align*}
      where
      \begin{align*}
        \lambda= {} & \frac {\Gamma(\frac {k - 1} 2 )}{\Gamma(\frac {n + s
            + 1} 2 )}\sum_{q=z}^{\lfloor \frac s 2 \rfloor}(-1)^{q -
          z} \binom{s }{2q} \Gamma(q + \tfrac 1 2) \frac {\Gamma(\frac
          {k + s + 1} 2 - q )\Gamma(\frac {n - k} 2 + q)}
        {\Gamma(\frac {n +1}2 + q)}
        \\
        & \times
        \!
        \begin{aligned}[t]
          \biggl[
          &\binom{q}{z}\frac{1}{2}(n-1)(k-1)-
          \binom{q-1}{z-1}(-1)q(k-1) -\binom{q-1}{z}(-1)q(k-1)\biggr]
        \end{aligned}
        \\
        = {} &\frac {\Gamma(\frac {k - 1} 2 )}{\Gamma(\frac {n + s + 1} 2
          )}\sum_{q=z}^{\lfloor \frac s 2 \rfloor}(-1)^{q - z}
        \binom{s }{2q} \Gamma(q + \tfrac 1 2) \frac {\Gamma(\frac {k +
            s + 1} 2 - q )\Gamma(\frac {n - k} 2 + q)} {\Gamma(\frac
          {n +1}2 + q)}\binom{q}{z}(k-1)\left(\tfrac{n-1}2 +q\right)
        \\
        = {} &2\frac {\Gamma(\frac {k +1} 2 )}{\Gamma(\frac {n + s + 1} 2
          )}\sum_{q=z}^{\lfloor \frac s 2 \rfloor}(-1)^{q - z}
        \binom{s }{2q}\binom{q}{z} \Gamma(q + \tfrac 1 2)\frac
        {\Gamma(\frac {k + s + 1} 2 - q )\Gamma(\frac {n - k} 2 + q)}
        {\Gamma(\frac {n -1}2 + q)},
      \end{align*}
      from which the assertion follows.
    \end{proof}
		
		\section{The Proofs for the Extrinsic Case}\label{14-sec5new}
		
		Our starting point is a relation, due to
  McMullen, which  relates the intrinsic and the extrinsic Minkowski tensors (see \cite[5.1 Theorem]{14-McMullen97}).
	Its proof can easily be localized (see \cite[Korollar 2.2.2]{14-Schuster03}). Combining this localization with the relation
	$Q=Q(E)+Q(E^\perp)$, where $E\subset \R^n$ is any $k$-flat, we obtain
	the following lemma.
	
  \begin{lemma} \label{14-Lem_McMullen_Int_Ext}
    Let $j, k, r, s \in \N_{0}$ with $j < k < n$, let $K \in \cK^{n}$ with $K \subset E \in A(n , k)$ and $\beta \in \cB(\R^{n})$.
		Then
    \begin{align*}
      \TenCM{j}{r}{s}{0} (K, \beta) & = \frac{\pi^{\frac {n - k} 2} s!}{\Gamma(\frac {n - j + s} 2)} \sum_{m = 0}^{\lfloor \frac s 2 \rfloor} \sum_{l = 0}^{m} (-1)^{m - l} \binom{m}{l} \frac {\Gamma(\frac {k - j + s} 2 - m)} {4^{m} \, m! (s - 2m)!}  Q^{l} Q(E)^{m - l} \IntTenCM{j}{r}{s - 2m}{0}{E} (K, \beta).
    \end{align*}
  \end{lemma}

		We start with the proof of Theorem \ref{14-Thm_ExtrCF_j=k-1}, for which
		we  use Theorem \ref{14-Thm_Main_Gen_Loc} after an
		application of Lemma \ref{14-Lem_McMullen_Int_Ext}.
  \begin{proof}[Proof (Theorem \ref{14-Thm_ExtrCF_j=k-1})]
    Lemma \ref{14-Lem_McMullen_Int_Ext} for $j=k-1$ gives
  \begin{align*}
    & \int_{A(n, k)} \TenCM{k - 1}{r}{s}{0} (K \cap E, \beta \cap E) \, \mu_k(\intd E) \\
    & \qquad = \frac{\pi^{\frac {n - k} 2} s!}{\Gamma(\frac {n - k + s + 1} 2)} \sum_{m = 0}^{\lfloor \frac s 2 \rfloor} \sum_{l = 0}^{m} (-1)^{m - l} \frac {\Gamma(\frac {s + 1} 2 - m)} {4^{m} \, m! (s - 2m)!} \binom{m}{l} Q^{l} \\
    & \qquad \qquad \times \int_{A(n, k)} Q(E)^{m - l} \IntTenCM{k - 1}{r}{s - 2m}{0}{E} (K \cap E, \beta \cap E) \, \mu_k(\intd E).
  \end{align*}
  For $j=k-1$ we can argue as in the proof of Corollary \ref{14-Cor_j=k-1} to see that Theorem \ref{14-Thm_Main_Gen_Loc} implies that
  \begin{align}
      & \int_{A(n, k)} Q(E)^{i} \IntTenCM{k - 1}{r}{s}{0}{E} (K \cap E, \beta \cap E) \, \mu_k(\intd E) \nonumber \\
      & \qquad = \gamma_{n, k, k - 1} \sum_{z = 0}^{\lfloor \frac {s} {2} \rfloor + i} \lambda_{n, k, k - 1, s, i, z} \, Q^z \TenCM{n - 1}{r}{s + 2i - 2z}{0} (K \cap E, \beta \cap E), \label{Form_CF_Main_j=k-1}
  \end{align}
  where
  \begin{align*}
    \lambda_{n, k, k - 1, s, i, z} & = (k - 1) \sum_{p = 0}^{i} \sum_{q = (z - p)^+}^{\lfloor \frac s 2 \rfloor + i - p} (-1)^{p + q - z} \binom{i}{p} \binom{s + 2i - 2p}{2q} \binom{p + q}{z} \\
    & \qquad\qquad \times \Gamma(q + \tfrac 1 2) \frac {\Gamma(\frac {k + s + 1} 2 + i - p - q)} {\Gamma(\frac {n + s + 1} 2 + i - p)} \frac {\Gamma(\frac {k - 1} 2 + p) \Gamma(\frac {n - k} 2 + q)} {\Gamma(\frac {n - 1} 2 + p + q)}.
  \end{align*}
	(Of course, for $i=0$ we recover Corollary \ref{14-Cor_j=k-1}.)
  Hence, we obtain
  \begin{align*}
    & \int_{A(n, k)} \TenCM{k - 1}{r}{s}{0} (K \cap E, \beta \cap E) \, \mu_k(\intd E) \\
    & \qquad = \gamma_{n, k, k - 1} \frac{\pi^{\frac {n - k} 2} s!}{\Gamma(\frac {n - k + s + 1} 2)} \sum_{m = 0}^{\lfloor \frac s 2 \rfloor} \sum_{l = 0}^{m} \sum_{z = 0}^{\lfloor \frac s 2 \rfloor - l} (-1)^{m - l} \frac {\Gamma(\frac {s + 1} 2 - m)} {4^{m} \, m! (s - 2m)!} \binom{m}{l} \\
    & \qquad \qquad \times \lambda_{n, k, k - 1, s - 2m, m - l, z} \, Q^{l + z} \TenCM{n - 1}{r}{s - 2l - 2z}{0} (K, \beta) .
  \end{align*}
  An index shift of the summation with respect to $z$ yields
  \begin{align*}
    & \int_{A(n, k)} \TenCM{k - 1}{r}{s}{0} (K \cap E, \beta \cap E) \, \mu_k(\intd E) \\
    & \qquad = \gamma_{n, k, k - 1} \frac{\pi^{\frac {n - k} 2} s!}{\Gamma(\frac {n - k + s + 1} 2)} \sum_{m = 0}^{\lfloor \frac s 2 \rfloor} \sum_{l = 0}^{m} \sum_{z = l}^{\lfloor \frac s 2 \rfloor} (-1)^{m - l} \frac {\Gamma(\frac {s + 1} 2 - m)} {4^{m} \, m! (s - 2m)!} \binom{m}{l} \\
    & \qquad \qquad \times \lambda_{n, k, k - 1, s - 2m, m - l, z - l} \, Q^{z} \TenCM{n - 1}{r}{s - 2z}{0} (K, \beta).
  \end{align*}
  Changing the order of summation gives
  \begin{align}
    & \int_{A(n, k)} \TenCM{k - 1}{r}{s}{0} (K \cap E, \beta \cap E) \, \mu_k(\intd E) \nonumber \\
    & \qquad = \gamma_{n, k, k - 1} \frac{\pi^{\frac {n - k} 2} s!}{\Gamma(\frac {n - k + s + 1} 2)} \sum_{z = 0}^{\lfloor \frac s 2 \rfloor} \sum_{l = 0}^{z} \sum_{m = l}^{\lfloor \frac s 2 \rfloor} (-1)^{m - l} \frac {\Gamma(\frac {s + 1} 2 - m)} {4^{m} \, m! (s - 2m)!} \binom{m}{l} \nonumber \\
    & \qquad \qquad \times \lambda_{n, k, k - 1, s - 2m, m - l, z - l} \, Q^{z} \TenCM{n - 1}{r}{s - 2z}{0} (K, \beta). \label{Form_CF_j=k-1}
  \end{align}
  The coefficients of the tensorial curvature measures on the right-hand side of \eqref{Form_CF_j=k-1} do not depend on the choice of $r \in \N_{0}$ or $\beta \in \cB(\R^{n})$. Thus, we can set
  \begin{align*}
    & \int_{A(n, k)} \TenCM{k - 1}{r}{s}{0} (K \cap E, \beta \cap E) \, \mu_k(\intd E) = \sum_{z = 0}^{\lfloor \frac s 2 \rfloor}
		\kappa_{n, k, s, z} \, Q^{z} \TenCM{n - 1}{r}{s - 2z}{0} (K, \beta),
  \end{align*}
	where the coefficient $\kappa_{n, k, s, z}$ is uniquely defined in the obvious way.
  By choosing $r = 0$ and $\beta = \R^{n}$, we can compare this to the Crofton formula for translation invariant Minkowski tensors in
	\cite{14-BernHug15}.
	In fact, since the functionals  $Q^{z} \TenCM{n - 1}{0}{s - 2z}{0} (K, \R^n)$, $z\in \{0,\ldots,\lfloor s/2\rfloor\}
	\setminus\{({s-1})/{2}\}$, are linearly independent, we can conclude from the Crofton formula for the translation invariant
	Minkowski tensors in \cite[Theorem 3]{14-BernHug15}
	that
  \begin{align*}
    \kappa_{n, k, s, z} = \frac{k - 1} {n - 1} \frac {\pi^{\frac {n - k} 2} \Gamma(\frac {n} 2)} {\Gamma(\frac {k} 2) \Gamma(\frac {n - k} 2)} \frac{\Gamma(\frac {s + 1} 2) \Gamma(\frac s 2 + 1)}{\Gamma(\frac{n - k + s + 1}{2}) \Gamma(\frac {n + s - 1} 2)} \frac{\Gamma(\tfrac {n - k} 2 + z) \Gamma(\tfrac {k + s - 1} 2 - z)}{\Gamma(\frac s 2 - z + 1) z!}
  \end{align*}
  for $z \neq ({s - 1})/ 2$. If $z = ( {s - 1} )/2$, then $\TenCM{n - 1}{0}{s - 2z}{0} (K, \R^{n}) = \MinTen{n - 1}{0}{1}(K) = 0$, and hence we do not get any information about the corresponding coefficient from the global theorem. Consequently, we have to calculate $\kappa_{n, k, s, ( {s - 1} )/2}$ directly, which is what we do later in the proof.

  But first we demonstrate that  the coefficients of the tensorial curvature measures in \eqref{Form_CF_j=k-1} can
	be determined also by a direct calculation if $s$ is even. In fact, we obtain
  \begin{align*}
    S & := \sum_{m = l}^{\lfloor \frac s 2 \rfloor} (-1)^{m - l} \frac {\Gamma(\frac {s + 1} 2 - m)} {4^{m} \, m! (s - 2m)!} \binom{m}{l} \lambda_{n, k, k - 1, s - 2m, m - l, z - l} \\
    & \phantom{:}= (k - 1) \sum_{m = l}^{\lfloor \frac s 2 \rfloor} \sum_{p = l}^{m} \sum_{q = (z - p)^+}^{\lfloor \frac s 2 \rfloor - p} (-1)^{m + l + p + q - z} \frac {\Gamma(\frac {s + 1} 2 - m)} {4^{m} \, m! (s - 2m)!} \\
    & \qquad \times \binom{m}{l} \binom{m - l}{p - l} \binom{s - 2p}{2q} \binom{p + q - l}{z - l} \Gamma(q + \tfrac 1 2) \\
    & \qquad \times \frac {\Gamma(\frac {k + s + 1} 2 - p - q)} {\Gamma(\frac {n + s + 1} 2 - p)} \frac {\Gamma(\frac {k - 1} 2 + p - l) \Gamma(\frac {n - k} 2 + q)} {\Gamma(\frac {n - 1} 2 + p + q - l)}.
  \end{align*}
  Changing the order of summation gives
  \begin{align*}
    S & = (k - 1) \sum_{p = l}^{\lfloor \frac s 2 \rfloor} \sum_{q = (z - p)^+}^{\lfloor \frac s 2 \rfloor - p} (-1)^{l + q - z} \binom{s - 2p}{2q} \binom{p + q - l}{z - l} \Gamma(q + \tfrac 1 2) \\
    & \qquad \times \frac {\Gamma(\frac {k + s + 1} 2 - p - q)} {\Gamma(\frac {n + s + 1} 2 - p)} \frac {\Gamma(\frac {k - 1} 2 + p - l) \Gamma(\frac {n - k} 2 + q)} {\Gamma(\frac {n - 1} 2 + p + q - l)} \\
    & \qquad \times \sum_{m = p}^{\lfloor \frac s 2 \rfloor} (-1)^{m + p} \binom{m}{l} \binom{m - l}{p - l} \frac {\Gamma(\frac {s + 1} 2 - m)} {4^{m} \, m! (s - 2m)!}.
  \end{align*}
    We denote the sum with respect to $m$ by $T$ and conclude
  \begin{align*}
    T & = \sum_{m = p}^{\lfloor \frac s 2 \rfloor} (-1)^{m + p} \binom{m}{l} \binom{m - l}{p - l} \frac {\Gamma(\frac {s + 1} 2 - m)} {4^{m} \, m! (s - 2m)!} \\
    & = \frac {1}{l! (p - l)!} \sum_{m = p}^{\lfloor \frac s 2 \rfloor} (-1)^{m + p} \frac {\Gamma(\frac {s + 1} 2 - m)} {4^{m} \, (m - p)! (s - 2m)!}.
  \end{align*}
  An index shift yields
  \begin{align*}
    T & = \frac {1}{2^{s} l! (p - l)!} \sum_{m = 0}^{\lfloor \frac s 2 \rfloor - p} (-1)^{m} \frac {2^{s - 2p - 2m} \Gamma(\frac {s + 1} 2 - p - m)} {m! (s - 2p - 2m)!}.
  \end{align*}
  Legendre's duplication formula gives
  \begin{align*}
    T & = \frac {\sqrt \pi}{2^{s} l! (p - l)!} \sum_{m = 0}^{\lfloor \frac s 2 \rfloor - p} (-1)^{m} \frac {1} {m! \Gamma(\frac {s} 2 - p - m + 1)}.
  \end{align*}
  If $s$ is even, the binomial theorem yields
  \begin{align*}
    T & = \frac {\sqrt \pi}{2^{s} l! (p - l)! (\frac s 2 - p)!} \sum_{m = 0}^{\frac s 2 - p} (-1)^{m} \binom{\frac s 2 - p}{m} \\
    & = \frac {\sqrt \pi}{2^{s} l! (p - l)! (\frac s 2 - p)!} (1 - 1)^{\frac s 2 - p} \\
    & = \1\{ p = \frac s 2 \} \frac {\sqrt \pi}{2^{s} l! (\frac s 2 - l)!}.
  \end{align*}
  Hence, we obtain
  \begin{align*}
    S & = \frac {(k - 1) \sqrt \pi}{2^{s} l! (\frac s 2 - l)!} \sum_{q = (z - \frac s 2)^+}^{0} (-1)^{l + q - z} \binom{\frac s 2 + q - l}{z - l} \Gamma(q + \tfrac 1 2) \frac {\Gamma(\frac {k + 1} 2 - q)} {\Gamma(\frac {n + 1} 2)} \frac {\Gamma(\frac {k + s - 1} 2 - l) \Gamma(\frac {n - k} 2 + q)} {\Gamma(\frac {n + s - 1} 2 + q - l)} \\
    & = (-1)^{l - z} \frac {(k - 1) \sqrt \pi \Gamma(\tfrac 1 2)}{2^{s} l! (\frac s 2 - l)!} \binom{\frac s 2 - l}{z - l} \frac {\Gamma(\frac {k + 1} 2)} {\Gamma(\frac {n + 1} 2)} \frac {\Gamma(\frac {k + s - 1} 2 - l) \Gamma(\frac {n - k} 2)} {\Gamma(\frac {n + s - 1} 2 - l)} \\
    & = (-1)^{l - z} \frac {\Gamma(\frac {k + 1} 2) \Gamma(\frac {n - k} 2)} {\Gamma(\frac {n + 1} 2)} \frac {(k - 1) \pi}{2^{s} l! (\frac s 2 - l)!} \binom{\frac s 2 - l}{z - l} \frac {\Gamma(\frac {k + s - 1} 2 - l)} {\Gamma(\frac {n + s - 1} 2 - l)}.
  \end{align*}
  Furthermore, Legendre's duplication formula yields
  \begin{align*}
    s! S & = (-1)^{l - z} \frac {(k - 1) \sqrt\pi \Gamma(\frac {k + 1} 2) \Gamma(\frac {n - k} 2) \Gamma(\frac {s + 1} 2)} {\Gamma(\frac {n + 1} 2)}\underbrace{\binom{\frac s 2}{l} \binom{\frac s 2 - l}{z - l}}_{ = \binom{\frac s 2}{z} \binom{z}{l}} \frac {\Gamma(\frac {k + s - 1} 2 - l)} {\Gamma(\frac {n + s - 1} 2 - l)}.
  \end{align*}
  Thus, we obtain
  \begin{align*}
    & \int_{A(n, k)} \TenCM{k - 1}{r}{s}{0} (K \cap E, \beta \cap E) \, \mu_k(\intd E) \\
    & \qquad = \gamma_{n, k, k - 1} \frac{\pi^{\frac {n - k + 1} 2}}{\Gamma(\frac {n - k + s + 1} 2)} \frac {(k - 1) \Gamma(\frac {k + 1} 2) \Gamma(\frac {n - k} 2) \Gamma(\frac {s + 1} 2)} {\Gamma(\frac {n + 1} 2)} \sum_{z = 0}^{ \frac s 2 } \binom{\frac s 2}{z} \\
    & \qquad \qquad \times \sum_{l = 0}^{z} (-1)^{l - z} \binom{z}{l} \frac {\Gamma(\frac {k + s - 1} 2 - l)} {\Gamma(\frac {n + s - 1} 2 - l)} \, Q^{z} \TenCM{n - 1}{r}{s - 2z}{0} (K, \beta).
  \end{align*}
  From Lemma \ref{14-Lem_Zeil_1} we conclude
  \begin{align*}
    & \int_{A(n, k)} \TenCM{k - 1}{r}{s}{0} (K \cap E, \beta \cap E) \, \mu_k(\intd E) \\
    & \qquad = \gamma_{n, k, k - 1} \frac{\pi^{\frac {n - k + 1} 2}}{\Gamma(\frac {n - k + s + 1} 2)} \frac {(k - 1) \Gamma(\frac {k + 1} 2) \Gamma(\frac {s + 1} 2)} {\Gamma(\frac {n + 1} 2) \Gamma(\frac {n + s - 1} 2)} \\
    & \qquad \qquad \times \sum_{z = 0}^{ \frac s 2 } \binom{\frac s 2}{z} \Gamma(\tfrac {k + s - 1} 2 - z) \Gamma(\tfrac {n - k} 2 + z) Q^{z} \TenCM{n - 1}{r}{s - 2z}{0} (K, \beta).
  \end{align*}
  With
  \begin{align*}
    \gamma_{n, k, k - 1} = \binom{n - 2}{k - 1} \frac{\Gamma(\frac {n - k + 1} 2)}{2 \pi} = \frac {(n - 2)!}{(n - k - 1)! (k - 1)!} \frac{\Gamma(\frac {n - k + 1} 2)}{2 \pi}
  \end{align*}
  we get
  \begin{align*}
    & \int_{A(n, k)} \TenCM{k - 1}{r}{s}{0} (K \cap E, \beta \cap E) \, \mu_k(\intd E) \\
    & \qquad = \frac {(n - 2)!}{\Gamma(\frac {n + 1} 2)} \frac {\Gamma(\tfrac {n - k + 1} 2) }{(n - k - 1)!} \frac {\Gamma(\frac {k + 1} 2)}{(k - 2)!} \frac{\pi^{\frac {n - k - 1} 2} \Gamma(\frac {s + 1} 2)}{2 \Gamma(\frac {n + s - 1} 2) \Gamma(\frac {n - k + s + 1} 2)} \\
    & \qquad \qquad \times \sum_{z = 0}^{ \frac s 2 } \binom{\frac s 2}{z} \Gamma(\tfrac {k + s - 1} 2 - z) \Gamma(\tfrac {n - k} 2 + z) Q^{z} \TenCM{n - 1}{r}{s - 2z}{0} (K, \beta).
  \end{align*}
  Legendre's formula applied three times gives
  \begin{align*}
    & \int_{A(n, k)} \TenCM{k - 1}{r}{s}{0} (K \cap E, \beta \cap E) \, \mu_k(\intd E) \\
    & \qquad = \frac {k - 1} {n - 1} \frac {\Gamma(\frac {n} 2)} {\Gamma(\frac {k} 2) \Gamma(\frac {n - k} 2)}\frac{\pi^{\frac {n - k} 2} \Gamma(\frac {s + 1} 2)}{\Gamma(\frac {n + s - 1} 2) \Gamma(\frac {n - k + s + 1} 2)} \\
    & \qquad \qquad \times \sum_{z = 0}^{ \frac s 2 } \binom{\frac s 2}{z} \Gamma(\tfrac {k + s - 1} 2 - z) \Gamma(\tfrac {n - k} 2 + z) Q^{z} \TenCM{n - 1}{r}{s - 2z}{0} (K, \beta),
  \end{align*}
  which confirms the coefficients for even $s$.

  On the other hand, if $s$ is odd, then Lemma \ref{14-Lem_Zeil_2} yields
  \begin{align*}
    T & = \frac {\sqrt \pi}{2^{s} l! (p - l)!} \sum_{m = 0}^{\frac {s - 1} 2 - p} (-1)^{m} \frac {1} {m! \Gamma(\frac {s} 2 - p - m + 1)} \\
    & = \frac {\sqrt \pi}{2^{s} l! (p - l)!} \bigg( \sum_{m = 0}^{\frac {s + 1} 2 - p} (-1)^{m} \frac {1} {m! \Gamma(\frac {s} 2 - p - m + 1)} - (-1)^{\frac {s + 1} 2 - p} \frac {1} {(\frac {s + 1} 2 - p)! \Gamma(\frac {1} 2)} \bigg) \\
    & = \frac {\sqrt \pi}{2^{s} l! (p - l)!} \bigg( (-1)^{\frac {s + 1} 2 - p} \frac{1}{\sqrt\pi (- s + 2p) (\frac {s + 1} 2 - p)!}- (-1)^{\frac {s + 1} 2 - p} \frac {1} {\sqrt \pi (\frac {s + 1} 2 - p)!} \bigg) \\
    & = (-1)^{\frac {s - 1} 2 - p} \frac {\sqrt \pi}{2^{s} l! (p - l)!} \frac{1}{\sqrt\pi (\frac {s + 1} 2 - p)!} (\tfrac 1 {s - 2p} + 1) \\
    & = (-1)^{\frac {s - 1} 2 - p} \frac {1}{2^{s - 1} (s - 2p) (\frac {s - 1} 2 - p)! l! (p - l)!} \\
    & = (-1)^{\frac {s - 1} 2 - p} \frac {2 \Gamma(\frac s 2 + 1)}{\sqrt \pi (s - 2p) s!} \binom{\frac{s - 1}{2}}{p} \binom{p}{l}.
  \end{align*}
  Hence, we obtain
  \begin{align*}
    s! \sum_{l = 0}^{z} S & = \frac {2 (k - 1) \Gamma(\frac s 2 + 1)}{\sqrt \pi} \sum_{l = 0}^{z} \sum_{p = l}^{\frac {s - 1} 2} \sum_{q = (z - p)^+}^{\frac {s - 1} 2 - p} (-1)^{\frac {s - 1} 2 + l + p + q - z} \frac {1}{(s - 2p)} \\
    & \qquad \times \binom{\frac{s - 1}{2}}{p} \binom{p}{l} \binom{s - 2p}{2q} \binom{p + q - l}{z - l} \Gamma(q + \tfrac 1 2) \\
    & \qquad \times \frac {\Gamma(\frac {k + s + 1} 2 - p - q)} {\Gamma(\frac {n + s + 1} 2 - p)} \frac {\Gamma(\frac {k - 1} 2 + p - l) \Gamma(\frac {n - k} 2 + q)} {\Gamma(\frac {n - 1} 2 + p + q - l)}.
  \end{align*}
  This yields
  \begin{align*}
    & \int_{A(n, k)} \TenCM{k - 1}{r}{s}{0} (K \cap E, \beta \cap E) \, \mu_k(\intd E) \\
    & \quad = 2 (k - 1) \gamma_{n, k, k - 1} \frac{\pi^{\frac {n - k - 1} 2} \Gamma(\frac s 2 + 1)}{\Gamma(\frac {n - k + s + 1} 2)} \sum_{z = 0}^{\frac {s - 1} 2} Q^{z} \TenCM{n - 1}{r}{s - 2z}{0} (K, \beta) \\
    & \qquad \times \sum_{l = 0}^{z} \sum_{p = l}^{\frac {s - 1} 2} \sum_{q = (z - p)^+}^{\frac {s - 1} 2 - p} (-1)^{\frac {s - 1} 2 + l + p + q - z} \frac {1}{(s - 2p)} \binom{\frac{s - 1}{2}}{p} \binom{p}{l} \binom{s - 2p}{2q} \binom{p + q - l}{z - l} \\
    & \qquad \times \Gamma(q + \tfrac 1 2) \frac {\Gamma(\frac {k + s + 1} 2 - p - q)} {\Gamma(\frac {n + s + 1} 2 - p)} \frac {\Gamma(\frac {k - 1} 2 + p - l) \Gamma(\frac {n - k} 2 + q)} {\Gamma(\frac {n - 1} 2 + p + q - l)}.
  \end{align*}
  With
  \begin{align*}
    \gamma_{n, k, k - 1} = \binom{n - 2}{k - 1} \frac{\Gamma(\frac {n - k + 1} 2)}{2 \pi} = \frac {(n - 2)!}{(n - k - 1)! (k - 1)!} \frac{\Gamma(\frac {n - k + 1} 2)}{2 \pi}
  \end{align*}
  we get
  \begin{align*}
    & \int_{A(n, k)} \TenCM{k - 1}{r}{s}{0} (K \cap E, \beta \cap E) \, \mu_k(\intd E) \\
    &  = \frac {(n - 2)!}{(n - k - 1)! (k - 2)!} \frac{\pi^{\frac {n - k - 3} 2} \Gamma(\frac {n - k + 1} 2) \Gamma(\frac s 2 + 1)}{\Gamma(\frac {n - k + s + 1} 2)} \sum_{z = 0}^{\frac {s - 1} 2} Q^{z} \TenCM{n - 1}{r}{s - 2z}{0} (K, \beta) \\
    & \quad  \times \sum_{l = 0}^{z} \sum_{p = l}^{\frac {s - 1} 2} \sum_{q = (z - p)^+}^{\frac {s - 1} 2 - p} (-1)^{\frac {s - 1} 2 + l + p + q - z} \frac {1}{(s - 2p)} \binom{\frac{s - 1}{2}}{p} \binom{p}{l} \binom{s - 2p}{2q} \binom{p + q - l}{z - l} \\
    & \quad \times \Gamma(q + \tfrac 1 2) \frac {\Gamma(\frac {k + s + 1} 2 - p - q)} {\Gamma(\frac {n + s + 1} 2 - p)} \frac {\Gamma(\frac {k - 1} 2 + p - l) \Gamma(\frac {n - k} 2 + q)} {\Gamma(\frac {n - 1} 2 + p + q - l)}.
  \end{align*}
  We denote the threefold sum with respect to $l$, $p$ and $q$ by $R$. Hence, $R$ multiplied with the factor in front of the sum with respect to $z$ equals $\kappa_{n, k, s, z}$. A direct calculation for $R$  still remains an open task. However, for the proof this is not required.

  Finally, if $s$ is odd we calculate the only so far unknown coefficient $\kappa_{n, k, s, ( {s - 1})/ 2}$. For $z = ( {s - 1})/ 2$ we see that the sum over $q$ only contains one summand, namely $q = ( {s - 1})/ 2 - p$. Hence, we obtain
  \begin{align*}
    R & = \Gamma(\tfrac {k} 2 + 1) \sum_{l = 0}^{z} \sum_{p = l}^{\frac {s - 1} 2} (-1)^{\frac {s - 1} 2 + l} \binom{\frac{s - 1}{2}}{p} \binom{p}{l} \Gamma(\tfrac s 2 - p) \frac {\Gamma(\frac {k - 1} 2 + p - l) \Gamma(\frac {n - k + s - 1} 2 - p)} {\Gamma(\frac {n + s + 1} 2 - p) \Gamma(\frac {n + s} 2 - l - 1)} \\
    & = \Gamma(\tfrac {k} 2 + 1) \sum_{p = 0}^{\frac {s - 1} 2} (-1)^{\frac {s - 1} 2} \binom{\frac{s - 1}{2}}{p} \Gamma(\tfrac s 2 - p) \frac {\Gamma(\frac {n - k + s - 1} 2 - p)} {\Gamma(\frac {n + s + 1} 2 - p)} \sum_{l = 0}^{p} (-1)^{l} \binom{p}{l} \frac {\Gamma(\frac {k - 1} 2 + p - l)} {\Gamma(\frac {n + s} 2 - l - 1)}.
  \end{align*}
  Then Lemma \ref{14-Lem_Zeil_1} yields
  \begin{align*}
    R & = \frac {\Gamma(\tfrac {k} 2 + 1) \Gamma(\frac {k - 1} 2) \Gamma(\frac {n - k + s - 1} 2)} {\Gamma(\frac {n + s} 2 - 1)} \sum_{p = 0}^{\frac {s - 1} 2} (-1)^{\frac {s - 1} 2 + p} \binom{\frac{s - 1}{2}}{p} \frac {\Gamma(\frac s 2 - p)} {\Gamma(\frac {n + s + 1} 2 - p)} \\
    & = \frac {\Gamma(\tfrac {k} 2 + 1) \Gamma(\frac {k - 1} 2) \Gamma(\frac {n - k + s - 1} 2)} {\Gamma(\frac {n + s} 2 - 1)} \sum_{p = 0}^{\frac {s - 1} 2} (-1)^{p} \binom{\frac{s - 1}{2}}{p} \frac {\Gamma(\frac 1 2 + p)} {\Gamma(\frac {n} 2 + 1 + p)}.
  \end{align*}
  Again, we apply Lemma \ref{14-Lem_Zeil_1} and obtain
  \begin{align*}
    R & = \sqrt \pi \frac {\Gamma(\tfrac {k} 2 + 1) \Gamma(\frac {k - 1} 2)} {\Gamma(\frac {n + 1} 2)} \frac {\Gamma(\frac {n + s} 2) \Gamma(\frac {n - k + s - 1} 2)} {\Gamma(\frac {n + s} 2 - 1) \Gamma(\frac {n + s + 1} 2) }.
  \end{align*}
  Thus, we conclude
  \begin{align*}
    \kappa_{n, k, s, \frac {s - 1} 2} & = \frac {(n - 2)!}{(n - k - 1)! (k - 2)!} \frac{\pi^{\frac {n - k - 3} 2} \Gamma(\frac {n - k + 1} 2) \Gamma(\frac s 2 + 1)}{\Gamma(\frac {n - k + s + 1} 2)} R \\
    & = \pi^{\frac {n - k - 2} 2} \frac {(n - 2)!}{\Gamma(\frac {n + 1} 2)} \frac {\Gamma(\tfrac {k} 2 + 1) \Gamma(\frac {k - 1} 2)}{(k - 2)!} \frac {\Gamma(\frac {n - k + 1} 2)}{(n - k - 1)!} \frac {(n + s - 2) \Gamma(\frac s 2 + 1)} {(n - k + s - 1) \Gamma(\frac {n + s + 1} 2) }.
  \end{align*}
  Applying three times Legendre's formula gives
  $$
    \kappa_{n, k, s, \frac {s - 1} 2} = \pi^{\frac {n - k - 1} 2}
		\frac {2k (n + s - 2)} {(n - 1) (n - k + s - 1)}
		\frac {\Gamma(\frac {n} 2)}{\Gamma(\frac {n - k} 2)} \frac {\Gamma(\frac s 2 + 1)} { \Gamma(\frac {n + s + 1} 2) },
  $$
	which completes the argument.
  \end{proof}

  Next we prove Theorem \ref{14-Thm_ExtrCF_k=1}. As in the previous proof, one can compare the Crofton integral to the global one obtained in \cite[Theorem 3]{14-BernHug15}. However, we deduce it directly from Theorem \ref{14-Thm_Main_k=1}.

  \begin{proof}[Proof (Theorem \ref{14-Thm_ExtrCF_k=1})]
    Lemma \ref{14-Lem_McMullen_Int_Ext} yields
    \begin{align*}
      & \int_{A(n, 1)} \TenCM{0}{r}{s}{0} (K \cap E, \beta \cap E) \, \mu_1(\intd E) \\
      & \qquad = \frac{\pi^{\frac {n - 1} 2} s!}{\Gamma(\frac {n + s} 2)} \sum_{m = 0}^{\lfloor \frac s 2 \rfloor} \sum_{l = 0}^{m} (-1)^{m - l} \binom{m}{l} \frac {\Gamma(\frac {s + 1} 2 - m)} {4^{m} \, m! (s - 2m)!} Q^{l} \\
      & \qquad \qquad \times \int_{A(n, 1)} Q(E)^{m - l} \IntTenCM{0}{r}{s - 2m}{0}{E} (K \cap E, \beta \cap E) \, \mu_1(\intd E).
    \end{align*}
    If $s \in \N_0$ is even, we conclude from Theorem \ref{14-Thm_Main_k=1}
    \begin{align*}
      & \int_{A(n, 1)} \TenCM{0}{r}{s}{0} (K \cap E, \beta \cap E) \, \mu_1(\intd E) \\
      & \qquad = \frac{\pi^{\frac {n - 1} 2} s!}{\Gamma(\frac {n + s} 2)} \sum_{m = 0}^{\frac s 2} \sum_{l = 0}^{m} (-1)^{m - l} \binom{m}{l} \frac {\Gamma(\frac {s + 1} 2 - m)} {4^{m} \, m! (s - 2m)!} \\
      & \qquad \qquad \times \frac {\Gamma( \frac n 2 ) \Gamma(\frac {s + 1} 2 - l)} { \pi \Gamma ( \frac {n + s + 1} 2 - l) } \sum_{z = 0}^{\frac {s} {2} - l} (-1)^{z} \binom{\frac s 2 - l}{z} \frac {1} {1 - 2z} \, Q^{\frac s 2 - z} \TenCM{n - 1}{r}{2z}{0} (K, \beta).
    \end{align*}
    A change of the order of summation yields
    \begin{align*}
      & \int_{A(n, 1)} \TenCM{0}{r}{s}{0} (K \cap E, \beta \cap E) \, \mu_1(\intd E) \\
      & \qquad = \frac{\pi^{\frac {n - 1} 2} s!}{\Gamma(\frac {n + s} 2)} \sum_{l = 0}^{\frac s 2} \sum_{m = l}^{\frac s 2} (-1)^{m - l} \binom{m}{l} \frac {\Gamma(\frac {s + 1} 2 - m)} {4^{m} \, m! (s - 2m)!} \\
      & \qquad \qquad \times \frac {\Gamma( \frac n 2 ) \Gamma(\frac {s + 1} 2 - l)} { \pi \Gamma ( \frac {n + s + 1} 2 - l) } \sum_{z = 0}^{\frac {s} {2} - l} (-1)^{z} \binom{\frac s 2 - l}{z} \frac {1} {1 - 2z} \, Q^{\frac s 2 - z} \TenCM{n - 1}{r}{2z}{0} (K, \beta).
    \end{align*}
    Legendre's duplication formula gives for the sum with respect to $m$, which we denote by $S$,
    \begin{align*}
      S & = \frac {\sqrt \pi} {2^{s}} \sum_{m = l}^{\frac s 2} (-1)^{m - l} \binom{m}{l} \frac {1} {m! \Gamma(\frac {s} 2 - m + 1)} \\
      & = \frac {\sqrt \pi} {2^{s} l!} \sum_{m = 0}^{\frac s 2 - l} (-1)^{m} \frac {1} {m! \Gamma(\frac {s} 2 - l - m + 1)}.
    \end{align*}
    As seen before, we conclude from the binomial theorem
    \begin{align*}
      S & = \frac {\sqrt \pi} {2^{s} (\frac s 2 - l)! l!} \sum_{m = 0}^{\frac s 2 - l} (-1)^{m} \binom{\frac s 2 - l}{m} \\
      & = \1\{ l = \frac s 2 \} \frac {\Gamma(\frac {s + 1} 2)} {s!}.
    \end{align*}
    Hence, we obtain
    \begin{align*}
      \int_{A(n, 1)} \TenCM{0}{r}{s}{0} (K \cap E, \beta \cap E) \, \mu_1(\intd E)
      &  = \frac{\pi^{\frac {n - 3} 2} \Gamma(\frac {s + 1} 2)}{\Gamma(\frac {n + s} 2)} \frac {\Gamma( \frac n 2 ) \Gamma(\frac {1} 2)} {\Gamma ( \frac {n + 1} 2) } Q^{\frac s 2} \TenCM{n - 1}{r}{0}{0} (K, \beta) \\
      & = \pi^{\frac {n - 2} 2}\frac {\Gamma( \frac n 2 ) \Gamma(\frac {s + 1} 2)} {\Gamma(\frac {n + s} 2) \Gamma ( \frac {n + 1} 2) } Q^{\frac s 2} \TenCM{n - 1}{r}{0}{0} (K, \beta).
    \end{align*}

    On the other hand, if $s \in \N$ is odd, we conclude from Theorem \ref{14-Thm_Main_k=1}
    \begin{align*}
      & \int_{A(n, 1)} \TenCM{0}{r}{s}{0} (K \cap E, \beta \cap E) \, \mu_1(\intd E) \\
      & \qquad = \frac{\pi^{\frac {n - 2} 2} \Gamma(\frac {n} 2) s!}{\Gamma(\frac {n + s} 2)} \sum_{m = 0}^{ \frac {s-1} 2 } \sum_{l = 0}^{m} (-1)^{m - l} \frac {\Gamma(\frac {s + 1} 2 - m)} {4^{m} \, m! (s - 2m)!} \binom{m}{l} \frac {\Gamma(\frac {s} 2 - l + 1)} {\Gamma(\frac {n + s + 1} 2 - l)} \, Q^{\frac {s - 1} 2} \TenCM{n - 1}{r}{1}{0} (K, \beta).
    \end{align*}
    A change of the order of summation yields
    \begin{align*}
      & \int_{A(n, 1)} \TenCM{0}{r}{s}{0} (K \cap E, \beta \cap E) \, \mu_1(\intd E) \\
      & \qquad = \frac{\pi^{\frac {n - 2} 2} \Gamma(\frac {n} 2) s!}{\Gamma(\frac {n + s} 2)} \sum_{l = 0}^{ \frac {s-1} 2 } \sum_{m = l}^{\frac {s-1} 2} (-1)^{m - l} \frac {\Gamma(\frac {s + 1} 2 - m)} {4^{m} \, m! (s - 2m)!} \binom{m}{l} \frac {\Gamma(\frac {s} 2 - l + 1)} {\Gamma(\frac {n + s + 1} 2 - l)} \, Q^{\frac {s - 1} 2} \TenCM{n - 1}{r}{1}{0} (K, \beta).
    \end{align*}
    Legendre's duplication formula gives for the sum with respect to $m$, which we denote by $S$,
    \begin{align*}
      S & = \frac {\sqrt \pi} {2^{s} l!} \sum_{m = 0}^{\frac {s-1} 2 - l} (-1)^{m} \frac {1} { m! \Gamma(\frac {s} 2 - l - m + 1)}.
    \end{align*}
    Then Lemma \ref{14-Lem_Zeil_2} yields
    \begin{align*}
      S & = \frac {\sqrt \pi} {2^{s} l!} \bigg( \sum_{m = 0}^{\frac {s + 1} 2 - l} (-1)^{m} \frac {1} { m! \Gamma(\frac {s} 2 - l - m + 1)} - (-1)^{\frac {s + 1} 2 - l} \frac {1} { (\frac {s + 1} 2 - l)! \Gamma(\frac {1} 2)} \bigg) \\
      & = \frac {\sqrt \pi} {2^{s} l!} \bigg( (-1)^{\frac {s - 1} 2 - l} \frac {1} { \sqrt \pi (s - 2l) (\frac {s + 1} 2 - l)!} - (-1)^{\frac {s + 1} 2 - l} \frac {1} { \sqrt \pi (\frac {s + 1} 2 - l)!} \bigg) \\
      & = (-1)^{\frac {s - 1} 2 - l} \frac {1} {2^{s - 1} l! (s - 2l) (\frac {s - 1} 2 - l)!}.
    \end{align*}
    Hence, we obtain
    \begin{align*}
      & \int_{A(n, 1)} \TenCM{0}{r}{s}{0} (K \cap E, \beta \cap E) \, \mu_1(\intd E) \\
      & \qquad = \frac{\pi^{\frac {n - 2} 2} \Gamma(\frac {n} 2) s!}{2^{s} \Gamma(\frac {n + s} 2)} \sum_{l = 0}^{\frac {s - 1} 2} (-1)^{\frac {s - 1} 2 - l} \frac {1} {l! (\frac {s - 1} 2 - l)!} \frac {\Gamma(\frac {s} 2 - l)} {\Gamma(\frac {n + s + 1} 2 - l)} \, Q^{\frac {s - 1} 2} \TenCM{n - 1}{r}{1}{0} (K, \beta) \\
      & \qquad = \frac{\pi^{\frac {n - 2} 2} \Gamma(\frac {n} 2) s!}{2^{s} \Gamma (\frac {s + 1} 2) \Gamma(\frac {n + s} 2)} \sum_{l = 0}^{\frac {s - 1} 2} (-1)^{l} \binom{\frac {s - 1} 2}{l} \frac {\Gamma(l + \frac {1} 2)} {\Gamma(\frac {n + 2} 2 + l)} \, Q^{\frac {s - 1} 2} \TenCM{n - 1}{r}{1}{0} (K, \beta).
    \end{align*}
    Then Lemma \ref{14-Lem_Zeil_1} gives
    \begin{align*}
      & \int_{A(n, 1)} \TenCM{0}{r}{s}{0} (K \cap E, \beta \cap E) \, \mu_1(\intd E) \\
      & \qquad = \frac{s!}{2^{s} \Gamma (\frac {s + 1} 2)} \frac {\pi^{\frac {n - 1} 2} \Gamma(\frac {n} 2)} {\Gamma(\frac {n + s + 1} 2) \Gamma(\frac {n + 1} 2)} \, Q^{\frac {s - 1} 2} \TenCM{n - 1}{r}{1}{0} (K, \beta).
    \end{align*}
    Finally, the assertion follows from Legendre's duplication formula.
  \end{proof}

Finally, we show that the Crofton formula has a very simple form in the $\psi$-representation of tensorial curvature measures.

 \begin{proof}[Proof of Corollary \ref{14-Cor_CF_Psi_Basis}]
The cases $s\in\{0,1\}$ are checked directly, hence we can assume $s\ge 2$ in the following. Using
    \eqref{14-Def_Psi_Basis} we get
    \begin{align}
      & \int_{A(n, k)} \PTenCM{k - 1}{r}{s}{0} (K \cap E, \beta \cap E)\,
			\mu_{k}(\intd E) \nonumber \\
      & \qquad = \frac 1 {\sqrt \pi} \sum_{j = 0}^{\lfloor \frac s 2 \rfloor}
			(-1)^{j} \binom{s}{2j} \frac{\Gamma(j + \frac 1 2)
			\Gamma(\frac n 2 + s - j - 1)}{\Gamma(\frac n 2 + s - 1)} Q^{j} \nonumber \\
      & \qquad \qquad \qquad \times \int_{A(n, k)} \TenCM{k - 1}{r}{s - 2j}{0}
			(K \cap E, \beta \cap E) \,\mu_{k}(\intd E). \label{14-Form_CF_Psi_Basis_1}
    \end{align}
    Then, for $k \neq 1$, Theorem~\ref{14-Thm_ExtrCF_j=k-1} yields
    \begin{align*}
      & \int_{A(n, k)} \PTenCM{k - 1}{r}{s}{0} (K \cap E, \beta \cap E)\, \mu_{k}(\intd E) \\
      & \qquad = \frac 1 {\sqrt \pi} \sum_{j = 0}^{\lfloor \frac s 2 \rfloor} (-1)^{j} \binom{s}{2j} \frac{\Gamma(j + \frac 1 2) \Gamma(\frac n 2 + s - j - 1)}{\Gamma(\frac n 2 + s - 1)} \sum_{z = 0}^{\lfloor \frac s 2 \rfloor - j} \kappa_{n, k, s - 2j, z} \, Q^{z + j} \TenCM{n - 1}{r}{s - 2j - 2z}{0} (K, \beta) \\
      & \qquad = \frac 1 {\sqrt \pi} \sum_{j = 0}^{\lfloor \frac s 2 \rfloor} \sum_{z = j}^{\lfloor \frac s 2 \rfloor} (-1)^{j} \binom{s}{2j} \frac{\Gamma(j + \frac 1 2) \Gamma(\frac n 2 + s - j - 1)}{\Gamma(\frac n 2 + s - 1)} \kappa_{n, k, s - 2j, z - j} \, Q^{z} \TenCM{n - 1}{r}{s - 2z}{0} (K, \beta),
    \end{align*}
    where
    \begin{align*}
      \kappa_{n, k, s - 2j, z - j} = \frac{k - 1} {n - 1} \frac {\pi^{\frac {n - k} 2} \Gamma(\frac {n} 2)} {\Gamma(\frac {k} 2) \Gamma(\frac {n - k} 2)} \frac{\Gamma(\frac {s + 1} 2 - j) \Gamma(\frac s 2 - j + 1)}{\Gamma(\frac{n - k + s + 1}{2} - j) \Gamma(\frac {n + s - 1} 2 - j)} \frac{\Gamma(\tfrac {n - k} 2 + z - j) \Gamma(\tfrac {k + s - 1} 2 - z)}{\Gamma(\frac s 2 - z + 1) (z - j)!},
    \end{align*}
    if $z \neq  (s - 1) /2$. On the other hand, if $z = (s - 1) /2$, then the coeffcient needs to be multiplied by the factor $\frac {k(n + s - 2j - 2)} {(k - 1)(n + s - 2j - 1)}$ (see the comment after the proof of Theorem \ref{14-Thm_ExtrCF_j=k-1}).

    Applying Legendre's duplication formula twice, we thus obtain
    \begin{align*}
      & \int_{A(n, k)} \PTenCM{k - 1}{r}{s}{0} (K \cap E, \beta \cap E)\,\mu_{k}(\intd E) \\
      & \qquad = \frac{k - 1} {n - 1} \frac {\pi^{\frac {n - k + 1} 2} \Gamma(\frac {n} 2)} {\Gamma(\frac {k} 2) \Gamma(\frac {n - k} 2)} \frac {s!} {2^{s}} \sum_{z = 0}^{\lfloor \frac s 2 \rfloor} \frac{\Gamma(\tfrac {k + s - 1} 2 - z)}{z! \Gamma(\frac n 2 + s - 1) \Gamma(\frac s 2 - z + 1)} Q^{z} \TenCM{n - 1}{r}{s - 2z}{0} (K, \beta) \\
      & \qquad \qquad \times \sum_{j = 0}^{z} (-1)^{j} \binom{z}{j} \frac{\Gamma(\frac n 2 + s - j - 1) \Gamma(\tfrac {n - k} 2 + z - j)}{\Gamma(\frac{n - k + s + 1}{2} - j) \Gamma(\frac {n + s - 1} 2 - j)} \\
      & \qquad \qquad \times \left( 1 - \1\{ z = \tfrac {s - 1} 2 \} \left(1 - \tfrac {k(n + s - 2j - 2)} {(k - 1)(n + s - 2j - 1)} \right) \right),
    \end{align*}
    Denoting the sum with respect to $j$ by $S_{z}$, an application of Lemma \ref{14-Lem_Zeil_jz} shows that
    \begin{align}
      S_{z} & =  \sum_{j = 0}^{z} (-1)^{j} \binom{z}{j} \frac{\Gamma(\frac n 2 + s - j - 1) \Gamma(\tfrac {n - k} 2 + z - j)}{\Gamma(\frac{n - k + s + 1}{2} - j) \Gamma(\frac {n + s - 1} 2 - j)} \nonumber \\
      & = (-1)^{z} \frac {\Gamma(\frac{n - k}{2}) \Gamma(\frac{s + 1}{2}) \Gamma(\frac{k + s - 1}{2}) \Gamma(\frac{n}{2} + s - z - 1)} {\Gamma(\frac{n - k + s + 1}{2}) \Gamma(\frac{n + s - 1}{2}) \Gamma(\frac{s + 1}{2} - z) \Gamma(\frac{k + s - 1}{2} - z)}, \label{Form_Cor_PsiB}
    \end{align}
    for $z \neq (s-1)/ 2$ and $k>1$. On the other hand, for $z = ( {s - 1} )/2 =: t$, we obtain from Lemma~\ref{14-Lem_Zeil_jz} and Lemma~\ref{14-Lem_Zeil_jt} (since $s > 1$ and thus $t > 0$) that
    \begin{align*}
      S_{t} & = \tfrac {k} {k - 1} \sum_{j = 0}^{t} (-1)^{j} \binom{t}{j} \left( 1 - \tfrac {1} {n + 2t - 2j} \right) \frac{\Gamma(\frac n 2 + 2t - j) \Gamma(\tfrac {n - k} 2 + t - j)}{\Gamma(\frac{n - k}{2} + t - j + 1) \Gamma(\frac {n} 2 + t - j)} \\
      & = \tfrac {k} {k - 1} \bigg( \sum_{j = 0}^{t} (-1)^{j} \binom{t}{j} \frac{\Gamma(\frac n 2 + 2t - j) \Gamma(\tfrac {n - k} 2 + t - j)}{\Gamma(\frac{n - k}{2} + t - j + 1) \Gamma(\frac {n} 2 + t - j)} \\
      & \qquad \qquad - \sum_{j = 0}^{t} (-1)^{j} \binom{t}{j} \tfrac {1} {\frac {n - k} 2 + t - j} \frac{\Gamma(\frac n 2 + 2t - j)}{\Gamma(\frac {n} 2 + t - j + 1)} \bigg) \\
      & = (-1)^{t} \frac{\Gamma(\frac{n - k}{2}) \Gamma(t + 1) \Gamma(\frac{k}{2} + t)}{\Gamma(\frac{k}{2}) \Gamma(\frac{n - k}{2} + t + 1)},
    \end{align*}
    which coincides with \eqref{Form_Cor_PsiB} for $z = ({s - 1})/ 2$.

    Thus, we have
    \begin{align*}
      & \int_{A(n, k)} \PTenCM{k - 1}{r}{s}{0} (K \cap E, \beta \cap E) \,\mu_{k}(\intd E) \\
      & \qquad = \frac{k - 1} {n - 1} \frac {\pi^{\frac {n - k + 1} 2} \Gamma(\frac {n} 2) \Gamma(\frac{k + s - 1}{2})} {\Gamma(\frac {k} 2) \Gamma(\frac{n - k + s + 1}{2}) \Gamma(\frac{n + s - 1}{2})} \frac {s! \Gamma(\frac{s + 1}{2})} {2^{s}} \\
      & \qquad \qquad \times \sum_{z = 0}^{\lfloor \frac s 2 \rfloor} (-1)^{z} \frac{\Gamma(\frac{n}{2} + s - z - 1)} {z! \Gamma(\frac n 2 + s - 1) \Gamma(\frac s 2 - z + 1) \Gamma(\frac{s + 1}{2} - z)} Q^{z} \TenCM{n - 1}{r}{s - 2z}{0} (K, \beta).
    \end{align*}
    Applying Legendre's duplication formula twice, we get
    \begin{align*}
      & \int_{A(n, k)} \PTenCM{k - 1}{r}{s}{0} (K \cap E, \beta \cap E)\, \mu_{k}(\intd E) \\
      & \qquad = \frac{k - 1} {n - 1} \frac {\pi^{\frac {n - k} 2} \Gamma(\frac {n} 2) \Gamma(\frac{k + s - 1}{2}) \Gamma(\frac{s + 1}{2})} {\Gamma(\frac {k} 2) \Gamma(\frac{n - k + s + 1}{2}) \Gamma(\frac{n + s - 1}{2})} \\
      & \qquad \qquad \times \frac 1 {\sqrt \pi} \sum_{z = 0}^{\lfloor \frac s 2 \rfloor} (-1)^{z} \binom{s}{2z} \frac{\Gamma(z + \frac 1 2)\Gamma(\frac{n}{2} + s - z - 1)} {\Gamma(\frac n 2 + s - 1)} Q^{z} \TenCM{n - 1}{r}{s - 2z}{0} (K, \beta).
    \end{align*}
    With \eqref{14-Def_Psi_Basis} we obtain the assertion for $k \neq 1$.

    On the other hand, if $k = 1$, then Theorem \ref{14-Thm_ExtrCF_k=1} yields for \eqref{14-Form_CF_Psi_Basis_1} that
    \begin{align*}
      & \int_{A(n, 1)} \PTenCM{0}{r}{s}{0} (K \cap E, \beta \cap E) \,\mu_{k}(\intd E) \\
      & \qquad = \frac {\pi^{\frac {n - 3} 2} \Gamma(\frac {n} 2)} {\Gamma(\frac {n + 1} 2) \Gamma(\frac n 2 + s - 1)} \sum_{j = 0}^{\lfloor \frac s 2 \rfloor} (-1)^{j} \binom{s}{2j} \frac{\Gamma(j + \frac 1 2) \Gamma(\frac n 2 + s - j - 1) \Gamma(\lfloor \frac {s + 1} 2 \rfloor - j + \frac 1 2)}{\Gamma(\frac {n} 2 + \lfloor \frac {s + 1} 2 \rfloor - j)} \\
      & \qquad \qquad \qquad \times Q^{\lfloor \frac s 2 \rfloor} \TenCM{n - 1}{r}{ s  - 2\lfloor \frac s 2 \rfloor}{0} (K, \beta).
    \end{align*}
    Denoting the sum with respect to $j$ by $S$ and applying Legendre's duplication formula three times, we conclude that
    \begin{align*}
      S = \sqrt \pi \Gamma(\lfloor \tfrac {s + 1} 2 \rfloor + \tfrac 1 2) \sum_{j = 0}^{\lfloor \frac s 2 \rfloor} (-1)^{j} \binom{\lfloor \frac s 2 \rfloor}{j} \frac{\Gamma(\frac n 2 + s - j - 1)}{\Gamma(\frac {n} 2 + \lfloor \frac {s + 1} 2 \rfloor - j)}.
    \end{align*}
   Since $s\ge 2$,  Lemma \ref{14-Lem_Zeil_1} yields $S = 0$ due to \eqref{14-proofLem7below}, and hence the assertion.
  \end{proof}

    \section{Sums of Gamma Functions}\label{14-sec5}

    In this section, we state four basic identities involving sums of Gamma
    functions.

    \begin{lemma} \label{14-Lem_Zeil_1} Let $q \in \N_0$ and $a, b >
      0$. Then
      \begin{equation*}
        \sum_{y=0}^q (-1)^{y} \binom{q}{y} \frac {\Gamma(a + y)} {\Gamma(b + y)} = \frac {\Gamma(a) \Gamma(b - a + q)} {\Gamma(b + q) \Gamma(b - a)}.
      \end{equation*}
    \end{lemma}

   Under the additional assumption $a<b$, this lemma can be found as Lemma 15.6.4 in \cite{14-AdleTayl07},
    which is also proved there. Since this case is not sufficient for our purposes, we deduce the
		current more general version via
    Zeilberger's algorithm.
		
		The factor $\Gamma(b - a + q)$ in Lemma \ref{14-Lem_Zeil_1} does not cause any problems in case $a - b - q \in \N_{0}$, as the also appearing $\Gamma(b - a)$ cancels out the singularity, see \eqref{14-proofLem7below}.

\begin{proof}
    We set
    \begin{align*}
       F(q, y) := (-1)^{y} \binom{q}{y} \frac {\Gamma(a + y)} {\Gamma(b + y)},
    \end{align*}
    for which we see that $F(q, y) = 0$ if $y \notin  \{ 0, \ldots, q \}$, and
    \begin{align*}
      f(q) := \sum_{y = 0}^{q} F(q, y).
    \end{align*}
    Furthermore, we define the function
    \begin{align*}
      G(q, y) :=
        \begin{cases}
          \frac{y (b + y - 1)}{q - y + 1} F(q, y), \qquad & \text{ for } y \in \{ 0, \ldots, q\}, \\
          G(q, q) -(b + q) F(q + 1, q) & \\
          \qquad + (b - a + q) F(q, q), \qquad & \text{ for } y = q + 1, \\
          0, & \text{ else}.
        \end{cases}
    \end{align*}
    A direct calculation yields
    \begin{align*}
      -(b + q - 1) F(q, y) + (b - a + q - 1) F(q - 1, y) = G(q - 1, y + 1) - G(q - 1, y)
    \end{align*}
    for $y \in \N_{0}$. Summing this relation over $y \in \{ 0, \ldots, q \}$ gives
    \begin{align*}
      - (b + q - 1) f(q) + (b - a + q - 1) f(q - 1) = 0
    \end{align*}
    and thus
    \begin{align*}
      f(q) & = \frac {b - a + q - 1} {b + q - 1} f(q - 1) \\
      & = \frac {(b - a + q - 2)(b - a + q - 1)} {(b + q - 2) (b + q - 1)} f(q - 2) \\
      & \ \, \vdots \\
      & = \frac {(b - a) \cdots (b - a + q - 1)} {b \cdots (b + q - 1)} f(0) \\
      & =  \frac {\Gamma(b - a + q) \Gamma(b)} {\Gamma(b + q) \Gamma(b - a)} f(0),
    \end{align*}
    where
		\begin{equation}\label{14-proofLem7below}
		\frac {\Gamma(b - a + q)} {\Gamma(b - a)} = (b - a) \cdots (b - a + q - 1)
		\end{equation}
		is well-defined, even for $a - b \in \N$.
    With
    \begin{align*}
      f(0) = \frac{\Gamma(a)}{\Gamma(b)}
    \end{align*}
    we obtain the assertion.
  \end{proof}

    \begin{lemma} \label{14-Lem_Zeil_2} Let $a \in
      \N_0$. Then\begin{equation*} \sum_{q = 0}^{a} \frac {(-1)^{q}
        }{\Gamma(a - q + \frac {1} 2) q!} = \frac{(-1)^{a}}{\sqrt \pi
          (1 - 2a)a !}.
      \end{equation*}
    \end{lemma}

  \begin{proof}
    For the sum $S$ on the left-hand side of the asserted equation, we
    obtain
    \begin{equation*}
      S = \sum_{q = 0}^{a} \left( \frac{2q}{2a - 1} \frac {(-1)^{q} }{\Gamma(a - q + \frac {1} 2) q!} + \frac{2q + 2}{2a - 1} \frac {(-1)^{q} }{\Gamma(a - q - \frac {1} 2) (q + 1)!} \right),
    \end{equation*}
    where we use that $(-\frac{1}{2})\Gamma(-\frac{1}{2})=\sqrt{\pi}$.
    Due to cancellation in this telescoping sum, the assertion follows
    immediately.
  \end{proof}

Finally, we establish the following lemmas.

  \begin{lemma} \label{14-Lem_Zeil_jz}
    Let $a, b, c \in \R$ and $z \in \N_{0}$ with $a > z\ge 0$ and $b > 0$. Then
    \begin{align*}
      & \sum_{j = 0}^{z} (-1)^{j} \binom{z}{j} \frac {\Gamma(a - j) \Gamma(b + z - j)} {\Gamma(c - j) \Gamma(a + b - c - j + 1)} \\
      & \qquad = (-1)^{z} \frac {\Gamma(a - z) \Gamma(b)} {\Gamma(a + b - c + 1) \Gamma(c)} \frac {\Gamma(a - c + 1)} {\Gamma(a - c +1 - z )} \frac {\Gamma(c - b)} {\Gamma(c - b - z)}.
    \end{align*}
  \end{lemma}

  The factor $\Gamma(a - c + 1)$ (resp. $\Gamma(c - b)$) in Lemma \ref{14-Lem_Zeil_jz} does not cause any problems for $c - a \in \N$ (resp. $b - c \in \N_{0}$), as the also appearing $\Gamma(a - c + 1 - z )$ (resp. $\Gamma(c - b - z)$) cancels out the singularity.
	On the other hand, in our applications of the lemma, we only need the cases where $a-c+1>z$ and $c-b>z$.

  \begin{proof}
    We set
    \begin{align*}
       F(z, j) := (-1)^{j} \binom{z}{j} \frac {\Gamma(a - j) \Gamma(b + z - j)} {\Gamma(c - j) \Gamma(a + b - c - j + 1)},
    \end{align*}
    for $j \in  \{ 0, \ldots, z \}$, and $F(z, j) = 0$ in all other cases, and
    \begin{align*}
      f(z) := \sum_{j = 0}^{z} F(z,j).
    \end{align*}
    Furthermore, we define the function
    \begin{align*}
      G(z, j) :=
        \begin{cases}
          - \frac{j (a - j) (b + z - j)}{z - j + 1} F(z, j), \qquad & \text{ for } j \in \{ 0, \ldots, z \}, \\
          G(z, z) + (a - z - 1) F(z + 1, z) & \\
          \qquad + (c - b - z - 1) (a - c - z) F(z, z), \qquad & \text{ for } j = z + 1, \\
          0, & \text{ otherwise}.
        \end{cases}
    \end{align*}
    A direct calculation yields
    \begin{align*}
      (a - z) F(z, j) + (c - b - z) (a - c - z + 1) F(z - 1, j) = G(z - 1, j + 1) - G(z - 1, j)
    \end{align*}
    for $j \in \N_{0}$. Summing this relation over $j \in \{ 0, \ldots, z \}$ gives
    \begin{align*}
      (a - z) f(z) + (c - b - z) (a - c - z + 1) f(z - 1) = 0
    \end{align*}
    and thus
    \begin{align*}
      f(z) & = - \frac {(c - b - z) (a - c - z + 1)} {a - z} f(z - 1) \\
      & = \frac {(c - b - z) (c - b - z + 1) (a - c - z + 1) (a - c - z + 2)} {(a - z)(a - z + 1)} f(z - 2) \\
      & \ \, \vdots \\
      & = (-1)^{z} \frac {(c - b - z) \cdots (c - b - 1) (a - c - z + 1) \cdots (a - c)} {(a - z) \cdots (a - 1)} f(0) \\
      & = (-1)^{z} \frac {\Gamma(c - b) \Gamma(a - c + 1) \Gamma(a - z)} {\Gamma(c - b - z) \Gamma(a - c+1 - z ) \Gamma(a)} f(0),
    \end{align*}
    where $$\frac {\Gamma(c - b)} {\Gamma(c - b - z)} = (c - b - z) \cdots (c - b - 1)$$ is well-defined, even for $b - c \in \N_{0}$,
		and a similar statement holds for  $ {\Gamma(a - c + 1)} /\Gamma(a - c +1- z )$.
    With
    \begin{align*}
      f(0) = \frac {\Gamma(a) \Gamma(b)} {\Gamma(c) \Gamma(a + b - c + 1)}
    \end{align*}
    we obtain the assertion.
  \end{proof}

  \begin{lemma} \label{14-Lem_Zeil_jt}
    Let $a, b \in \R$ with $a, b > 0$ and $t \in \N$. Then
    \begin{align*}
      & \sum_{j = 0}^{t} (-1)^{j} \frac{1}{b + j} \binom{t}{j} \frac {\Gamma(a + t + j)} {\Gamma(a + 1 + j)} = \frac {\Gamma(a - b + t) \Gamma(b) \Gamma(t + 1)} {\Gamma(a - b + 1) \Gamma(b + t + 1)}.
    \end{align*}
  \end{lemma}
  The factor $\Gamma(a - b + t)$ in Lemma \ref{14-Lem_Zeil_jt} does not cause any problems for $b - a - t \in \N_{0}$, as the also appearing $\Gamma(a - b + 1)$ cancels out the singularity. In our application of the lemma, we will additionally know that $a>b$.

  \begin{proof}
    We set
    \begin{align*}
       F(t, j) := (-1)^{j} \frac{1}{b + j} \binom{t}{j} \frac {\Gamma(a + t + j)} {\Gamma(a + 1 + j)},
    \end{align*}
    for which we see that $F(t, j) = 0$ if $j \notin  \{ 0, \ldots, t \}$, and
    \begin{align*}
      f(t) := \sum_{j = 0}^{t} F(t, j).
    \end{align*}
    Furthermore, we define the function
    \begin{align*}
      G(t, j) :=
        \begin{cases}
          \frac{j (a + j) (a + 2t + 1) (t^2 + t(a + 1) - j + 1) (b + j)}{t (t - j + 1) (a + t) (a + t + 1)} F(t, j), \qquad & \text{ for } j \in \{ 0, \ldots, t \}, \\
          G(t, t) - (b + t + 1) F(t + 1, t) & \\
          \qquad + (t + 1) (a - b + t) F(t, t), \qquad & \text{ for } j = t + 1, \\
          0, & \text{ otherwise}.
        \end{cases}
    \end{align*}
    A direct calculation yields
    \begin{align*}
      - (b + t) F(t, j) + t (a - b + t - 1) F(t - 1, j) = G(t - 1, j + 1) - G(t - 1, j)
    \end{align*}
    for $j \in \N_{0}$. Summing this relation over $j \in \{ 0, \ldots, t \}$ gives
    \begin{align*}
      - (b + t) f(t) + t (a - b + t - 1) f(t - 1) = 0
    \end{align*}
    and thus
    \begin{align*}
      f(t) & = \frac {t (a - b + t - 1)} {b + t} f(t - 1) \\
      & = \frac {(t - 1) t (a - b + t - 2) (a - b + t - 1)} {(b + t - 1) (b + t)} f(t - 2) \\
      & \ \, \vdots \\
      & = \frac {2 \cdots t (a - b + 1) \cdots (a - b + t - 1)} {(b + 2) \cdots (b + t)} f(1) \\
      & =  \frac {\Gamma(t + 1) \Gamma(a - b + t) \Gamma(b + 2)} {\Gamma(a - b + 1) \Gamma(b + t + 1)} f(1).
    \end{align*}
    With
    \begin{align*}
      f(1) = \frac{1}{b} - \frac{1}{b + 1} = \frac{1}{b (b + 1)}
    \end{align*}
    we obtain the assertion.
  \end{proof}

\end{document}